\documentclass[11pt,a4paper]{article}
\usepackage[margin=0.82in]{geometry}
\usepackage{amsmath,amssymb,amsthm,mathtools}
\usepackage{microtype}
\usepackage[hidelinks]{hyperref}
\usepackage{needspace}

\allowdisplaybreaks[2]
\newtheorem{theorem}{Theorem}[section]
\newtheorem{lemma}[theorem]{Lemma}
\newtheorem{proposition}[theorem]{Proposition}
\newtheorem{corollary}[theorem]{Corollary}
\theoremstyle{remark}

\newcommand{\rk}{\operatorname{rank}}
\newcommand{\msr}{\operatorname{msr}}
\newcommand{\F}{\mathbb F}

\title{Triangle-free Rank--Ramsey graphs from ordered incidence lifts\thanks{This work is supported by the National Natural Science Foundation of China (Nos. 12371348 and 12201258) and the High-Quality Science and Technology Cultivation Project of Jiangsu Normal University (No. JSNUGZL2026069).}}
\author{Yong Lu\thanks{Corresponding author.}\quad Qi Wu\quad Qiannan Zhou\\[2mm]
\small School of Mathematics and Statistics, Jiangsu Normal University,\\
\small Xuzhou, Jiangsu 221116, People's Republic of China\\
\small E-mails: \texttt{luyong@jsnu.edu.cn}, \texttt{wuqimath@163.com}, \texttt{qnzhoumath@163.com}}
\date{}
\begin{document}
\maketitle
\begin{abstract}
For a graph $G$ with adjacency matrix $A_G$ and a coefficient field
$\mathbb K$, put $\rho_{\mathbb K}(G)=\operatorname{rank}_{\mathbb K}(A_G+I)$.
We answer the triangle-free Rank--Ramsey existence question of
Beniamini, Linial and Shraibman. For every order $n$, we construct a
triangle-free graph $H_n$ such that, simultaneously over all coefficient fields,
\[
 \rho_{\mathbb K}(H_n)=\bigl(4(2/3)^{3/4}+o(1)\bigr)n^{3/4},
\]
with an error uniform in the field. The graphs can be chosen Hamiltonian
for all sufficiently large $n$. Our method combines ordered layers of
point--line incidence flags with a common low-dimensional endpoint space.
We determine the exact real rank, independence number and chromatic
number of the affine lifts. For all triangle-free complete incidence lifts,
we obtain a sharp two-term lower bound for complete incidence lifts,
with sharpness witnessed by projective-plane incidence graphs over
suitable coefficient fields. We also establish a third-order estimate
and stability near equality. On the same
explicit graphs, both the Lov\'asz number and minimum semidefinite rank
exceed the prescribed Boolean-matrix rank by polynomial factors. Thus neither parameter has a
universal constant-factor upper bound in terms of complement rank. We determine
Shannon capacity in the short-layer range.
\end{abstract}

\smallskip
{\small
\par}
\medskip

\section{Introduction and main results}

In the Rank--Ramsey problem of Beniamini et al.\ \cite{BLS}, we
seek graphs with small clique number and small shifted adjacency rank.
For a finite simple graph $G$, write $V(G)$ and $E(G)$ for its vertex
and edge sets, $\overline G$ for its complement, and $A_G$ for its
adjacency matrix. The order of $G$ is $|V(G)|$. We use $I$ for identity matrices, $J$ for all-ones
matrices, and $\mathbf1$ for all-ones vectors; subscripts specify sizes
or indexing sets. The superscript $\mathsf T$ denotes transpose.
Over a field $\mathbb K$, put
$\rho_{\mathbb K}(G)=\rk_{\mathbb K}(A_G+I)$, and abbreviate
$\rho_{\mathbb R}(G)$ to $\rho(G)$. We call this \emph{complement
rank}, since $A_{\overline G}=J-(A_G+I)$ and the two ranks differ
by at most one.

Let $\omega(G)$ be the largest size of a clique in $G$.
For integers $n,d\ge1$ and a field $\mathbb K$, define
\[
 \nu_{d,\mathbb K}(n)=\min\{\rho_{\mathbb K}(G):|V(G)|=n,\ \omega(G)\le d\},
 \qquad \nu_d(n)=\nu_{d,\mathbb R}(n).
\]
Let $\alpha(G)$ be the largest size of an independent set, and let
$\chi(G)$ be the minimum number of colors in a proper vertex coloring.
An independent set indexes an identity principal submatrix of $A_G+I$,
so $\alpha(G)\le\rho_{\mathbb K}(G)$ over every field.
Thus the rank condition implies an ordinary Ramsey bound, although
small independence number alone does not control rank.
The results of Ajtai et al.\ \cite{AKS} and Kim~\cite{Kim}
show that the smallest independence number among $n$-vertex
triangle-free graphs has order $\sqrt{n\log n}$. Controlling
$\rho(G)$ is a stronger requirement: one must bound the rank of
the prescribed Boolean matrix $A_G+I$, not only exclude large
independent sets.
We write $f\sim g$ when $f/g\to1$ and $f\asymp g$ when this
ratio stays bounded above and away from zero. We use the usual
$O$, $o$ and $\Omega$ notation. An estimate is \emph{uniform in the
coefficient field} when its error bound is independent of that field.
All logarithms are natural.

The motivation comes from communication complexity. For a Boolean
function $f:X\times Y\to\{0,1\}$ on finite sets, put
$M_f=(f(x,y))$ and $r=\operatorname{rank}_{\mathbb R}M_f$.
Let $D(f)$ be the minimum worst-case number of bits exchanged by a
deterministic protocol computing $f(x,y)$, when the parties know
$x$ and $y$ separately.
Lov\'asz and Saks~\cite{LovaszSaks} conjectured a bound polynomial
in $\log(r+1)$. Lovett~\cite{Lovett} proved
$D(f)=O(\sqrt r\log(r+1))$, and Sudakov and
Tomon~\cite{SudakovTomon} improved this to $O(\sqrt r)$.
The Rank--Ramsey problem asks for the corresponding clique
structure in low-rank Boolean matrices~\cite{BLS}.

Beniamini et al.\ \cite[Theorem~1.1 and Open Problem~8.1]{BLS}
constructed Rank--Ramsey graphs with a fixed clique bound as small
as $41$ and asked whether $\nu_2(n)=O(n^{1-c})$ for some $c>0$.
Using finite geometry, Bamberg et al.\ \cite{BBIR} obtained
triangle-free graphs giving the upper bound
$\nu_2(n)\le(1/4+o(1))n$ along an unbounded sequence of orders. Our exponent $3/4$ improves the order of growth,
not just this linear coefficient.

Codenotti et al.\ \cite{CPR} used directed incidence graphs and
ordered symmetrization to construct ordinary triangle-free Ramsey
graphs; Beniamini et al.\ \cite[Appendix~A]{BLS} analyzed the
relevant shifted adjacency matrix. Kostochka et al.\ \cite[Lemma~2.2]{KPR} proved that
an ordered superline graph on the edges of a bipartite base has
independence number smaller than the number of incidence vertices,
for every ordering. Kocbek~\cite{Kocbek} developed related geometric Ramsey
constructions, controlling independence numbers rather than the rank
of the fixed matrix $A_G+I$. As recalled in Beniamini et al.\ \cite[Appendix~A]{BLS},
the earlier construction yields triangle-free graphs satisfying
$\operatorname{rank}_{\mathbb R}(A_G-2I)=O(|V(G)|^{3/4})$.
This is a different, non-Boolean matrix shift; it does not bound
$\operatorname{rank}_{\mathbb R}(A_G+I)$.
Our new step is to use ordered copies so that
the low-dimensional row and column spaces of the directed shifted
matrix are shared across the layers, rather than counted afresh in
each layer.

An incidence flag is an edge of a bipartite incidence graph,
viewed as an incident point--line pair. We use independent, ordered
layers of flag copies. Between two layers, a flag is joined to
exactly those later flags to which it points in the directed
incidence graph. The endpoint factorization from \cite{CPR} then controls
its complement rank.
A \emph{transitive directed triangle} consists of distinct vertices
$x,y,z$ and arcs $x\to y$, $y\to z$, $x\to z$; extra reverse arcs
are allowed. A triangle in the lift requires this pattern in the base.
If the bipartite incidence graph contains no four-cycle as a subgraph,
which we abbreviate to \emph{$C_4$-free}, this pattern is absent.

For $B=(L\cup R,E)$, its directed flag graph $D_B$ has vertex set $E$:
$e=(u,v)$ points to $f=(u',v')$ exactly when $u\ne u'$, $v\ne v'$
and $uv'\in E$. Let $m=|E|$, let $k$ count the non-isolated vertices,
and let $c$ count the components containing an edge. For an integer
$t\ge1$, the complete ordered lift $\Gamma(B,t)$ has $t$ independent
layers of all $m$ flags; for $i<j$, the copies $(e,i)$ and $(f,j)$ are
adjacent exactly when $e\to f$ in $D_B$. Its order is $mt$.
Section~2 proves $\rho_{\mathbb K}(\Gamma(B,t))\le m+(t-1)(k-c)$.
``Complete'' refers to retaining all flags in every layer, not to a
complete graph.

For a prime power $q$, let $\F_q$ be a field with $q$ elements.
Let $B_q$ be the incidence graph of points $(x,y)\in\F_q^2$ and
lines $y=ax+b$ with $a,b\in\F_q$, and put $G(q,t)=\Gamma(B_q,t)$.
The geometric field $\F_q$ and the coefficient field $\mathbb K$
used to compute rank need not have the same characteristic.
There are $q^3$ flags and $2q^2$ incidence vertices. The general
bound therefore has size $q^3+O(tq^2)$, whereas the lift has $tq^3$
vertices. Taking $t$ proportional to $q$ balances the rank terms
and gives the exponent $3/4$.

\begin{theorem}[Affine lifts]\label{thm:affine}
For every prime power $q\ge2$ and integer $t\ge1$, the graph
$G(q,t)$ is triangle-free, has order $q^3t$, and is regular of
degree $(q-1)^2(t-1)$. Over every field $\mathbb K$,
\[
 \rho_{\mathbb K}(G(q,t))\le q^3+(2q^2-1)(t-1).
\]
Moreover, for every $q\ge2$ and $t\ge1$,
\[
 \alpha(G(q,t))=\max\{q^3,tq^2\},\qquad
 \chi(G(q,t))=\min\{q,t\}.
\]
The rank bound is an equality over $\mathbb R$ when $q\ge3$.
If $q\ge3$ and $t\ge3$, then $G(q,t)$ has a Hamilton cycle, that is,
a cycle containing every vertex.
\end{theorem}

Choosing the number of layers proportional to $q$ and passing to
induced subgraphs answers the triangle-free existence question.
The same construction also separates complement rank from two
other graph parameters.

Let $\vartheta(G)$ be the Lov\'asz number, the standard semidefinite
upper bound on $\alpha(G)$ introduced by Lov\'asz~\cite{Lovasz}.
A real symmetric matrix $M$ is positive semidefinite if
$x^{\mathsf T}Mx\ge0$ for every real vector $x$; equivalently, it is a
Gram matrix of real vectors. Following Beniamini et al.\ \cite{BLS},
define $\msr(G)$ as its minimum possible rank when the diagonal entries
are one and entries at distinct nonadjacent vertices are zero.
Entries at edges may also be zero. This convention matters:
$A_G+I$ need not be positive semidefinite.

Haemers' minrank~\cite{Haemers} minimizes rank over matrices
with nonzero diagonal and zeros at nonedges, over the same field;
it is at most $\rho_{\mathbb K}(G)$. Beniamini et al.\
\cite[Section~6 and Open Problem~8.2]{BLS} asked about the relation
between the Lov\'asz number and complement rank. Their
Theorem~6.2 exhibits a triangle-free family with unbounded
$\rho/\vartheta$. We obtain the reverse polynomial separation,
$\vartheta/\rho\to\infty$, on triangle-free graphs.
Thus the two parameters cannot be compared in either direction
by a universal constant.

\Needspace{15\baselineskip}
\begin{theorem}[Triangle-free graphs at every order]
\label{thm:main}
Put $C_0=3$ and $C_*=4(2/3)^{3/4}$. For every integer
$n\ge1$, there is a triangle-free graph $H_n$
of order $n$ with $\msr(H_n)=\lceil n/2\rceil$ and
$\rho_{\mathbb K}(H_n)\le C_0n^{3/4}$ over every field. The graphs can be chosen Hamiltonian for every
sufficiently large $n$. As $n\to\infty$,
\[
 \rho_{\mathbb K}(H_n)=(C_*+o(1))n^{3/4},
 \qquad
 \vartheta(H_n)=\left(\tfrac12(3/2)^{1/8}+o(1)\right)n^{7/8},
\]
with rank error uniform in $\mathbb K$.
\end{theorem}

The graphs $H_n$ are chosen independently of the coefficient field.
Finite-field ranks may differ at particular orders; the assertion
is that their leading asymptotic coefficient is common, with an error
that is uniform in the field.
Thus $\nu_2(n)=O(n^{3/4})$ and, on the same graphs,
$\vartheta(H_n)/\rho(H_n)\asymp n^{1/8}$ and
$\msr(H_n)/\rho(H_n)\asymp n^{1/4}$.
Section~5 records the main parameter consequences and distinguishes
new calculations for the lifts from the classical inequalities used
to obtain them. Large theta--independence gaps without the
triangle-free condition were previously studied by Feige~\cite{Feige}.

For complete incidence lifts, we determine the sharp leading and
second-order terms of a uniform rank lower bound, even for admissible
incidence bases containing four-cycles.

\begin{theorem}[Two sharp terms and a third-order estimate]
\label{thm:methodconstant}
Put $D_*=5\sqrt{2/3}$ and $E_+=\frac{13}{6}(2/3)^{1/4}$.
There is an absolute constant $C$ such that every triangle-free
complete incidence lift $\Gamma(B,t)$ of order $N=t|E(B)|\ge1$
over a simple bipartite graph $B$ satisfies
\[
 \rho_{\mathbb K}(\Gamma(B,t))
 \ge C_*N^{3/4}-D_*N^{1/2}-E_+N^{1/4}-C
\]
over every field $\mathbb K$.
Both $C_*$ and $D_*$ are sharp for a bound uniform over all incidence
bases and coefficient fields. One extremal sequence uses the point--line incidence graph of the
projective plane over $\F_q$, with $q=11^{2j+1}$, $t=(3q-5)/2$,
and coefficient field $\F_5$.
\end{theorem}

The global bounds remain $\Omega(\sqrt{n\log n})\le\nu_2(n)\le
O(n^{3/4})$, by the classical bound~\cite{AKS} and
Theorem~\ref{thm:main}.
In Section~2, we introduce ordered incidence lifts and establish the
rank and independence tools used throughout. In Section~3, we analyze
biregular incidence bases and determine the affine ranks. In Section~4,
we prove the lower bounds and stability results for complete lifts.
In Section~5, we study the Lov\'asz number, Hamilton cycles and
related graph parameters.

\section{Directed incidence graphs and ordered lifts}

For a graph $G$ and vertex $v$, let $N_G(v)$ be its neighborhood
and $d_G(v)=|N_G(v)|$ its degree. Write $\delta(G)$ and $\Delta(G)$
for the minimum and maximum degrees, and $c(G)$ for its number of
connected components. For $S\subseteq V(G)$, write $G[S]$ for the induced
subgraph and $G-S=G[V(G)\setminus S]$. We use $C_r$ for a cycle of
length $r$ and $K_{a,b}$ for a complete bipartite graph with part sizes
$a,b$. Let $B=(L\cup R,E)$ be a
simple bipartite graph with $|E|=m$. Every edge is written as
$(u,v)$, where $u\in L$ and $v\in R$.
Digraphs have at most one arc in each direction between two vertices.
Recall the directed flag graph $D_B$ from the introduction.
For two flags $e=(u,v)$ and $f=(u',v')$ with $u\ne u'$ and
$v\ne v'$, the original edges of $B$ are $uv$ and $u'v'$.
The arc $e\to f$ is present precisely when the cross edge $uv'$ also
belongs to $E$. The reverse arc $f\to e$ would require the other
cross edge $u'v$. Thus opposite arcs occur exactly when the four
edges $uv$, $u'v'$, $uv'$, and $u'v$ form a four-cycle in $B$.

\Needspace{9\baselineskip}
\begin{lemma}[The forbidden directed pattern]\label{lem:transitive}
The graph $D_B$ has no opposite arcs if and only if $B$ is
$C_4$-free. It has no transitive directed triangle if and only if
no four-cycle of $B$ has an edge whose two endpoints both have
degree at least three in $B$. In particular, a $C_4$-free base
satisfies both conditions.
\end{lemma}
\begin{proof}
The first equivalence follows from the cross-edge description
above.

Suppose the flags $e_i=(u_i,v_i)$ give the arcs
$e_1\to e_2$, $e_1\to e_3$ and $e_2\to e_3$.
The vertices $u_1,u_2,u_3$ are distinct, as are $v_1,v_2,v_3$.
The vertices $u_1,u_2,v_2,v_3$ form a four-cycle, and the adjacent
vertices $u_1,v_3$ have the additional neighbors $v_1,u_3$.
Conversely, label a four-cycle by $u_1,u_2,v_2,v_3$ so that
the adjacent vertices $u_1,v_3$ both have degree at least three.
Choose a neighbor $v_1$ of $u_1$ outside $\{v_2,v_3\}$ and
a neighbor $u_3$ of $v_3$ outside $\{u_1,u_2\}$.
The flags $(u_1,v_1),(u_2,v_2),(u_3,v_3)$ form the required
transitive directed triangle.
\end{proof}

We write the factorization of Codenotti et al.\ \cite{CPR} in
endpoint coordinates and record its common row and column space.

For a matrix $M$, write $\operatorname{col}M$ for its column space
and $\ker M$ for its kernel.
For invertible $M$, put $M^{-\mathsf T}=(M^{-1})^{\mathsf T}$.
We write $\operatorname{span}$ for linear span and $W^\perp$ for the
orthogonal complement of a real subspace $W$. Let $A$ be the $m\times m$ adjacency matrix of $D_B$.
Let $N$ be the $|L|\times|R|$ bipartite incidence matrix of $B$.
Let $X$ and $Y$ be the $m\times|L|$ and $m\times|R|$ matrices
recording, respectively, the left and right endpoints of every
edge. Thus each row of $X$, and each row of $Y$, contains exactly
one entry equal to one. Put $P_L=XX^{\mathsf T}-I_m$ and
$P_R=YY^{\mathsf T}-I_m$.

\begin{lemma}[The endpoint factorization]
\label{lem:factor}
Put $C=[X\ Y]$, $W=\operatorname{col}C$, and
$Q=\left(\begin{smallmatrix}-I&N\\0&-I\end{smallmatrix}\right)$.
Over every field,
\[
 A=P_LP_R,\qquad
 A-I_m=CQC^{\mathsf T}
      =XNY^{\mathsf T}-XX^{\mathsf T}-YY^{\mathsf T}.
\]
Both $A-I_m$ and $A^{\mathsf T}-I_m$ have their columns in $W$
and vanish on $\ker C^{\mathsf T}$. After discarding isolated
vertices, if $B$ has $k$ vertices and $c$ components, then
$\dim W=k-c$. Over $\mathbb R$, $\ker C^{\mathsf T}=W^\perp$.
\end{lemma}
\begin{proof}
The $(e,f)$ entry of $P_LP_R$ counts the cross edge $uv'$ when
it is distinct from both flags, precisely the condition for $e\to f$.
Expanding with $X^{\mathsf T}Y=N$ gives the factorization and the
common-space assertions. Finally, for vectors $a$ on $L$ and $b$ on $R$,
$C(a,b)^{\mathsf T}=0$ means $a_u+b_v=0$ on every edge. In each connected component, all left
coordinates therefore equal one scalar and all right coordinates
its negative. There is one free scalar per component over every
field, giving rank $k-c$.
\end{proof}

The real shifted rank is also related to non-backtracking spectra:
for the two orientations of the edges of $B$, the non-backtracking
matrix is $T_B=\left(\begin{smallmatrix}0&P_R\\P_L&0\end{smallmatrix}\right)$,
and $x\mapsto(P_Rx,x)$ identifies $\ker(A-I)$ with
$\ker(T_B-I)$. See Torres~\cite[Section~2]{Torres} for the
unit eigenvalue. We retain a short joint proof because we also need
the common column space and the positive-characteristic cases.

\Needspace{9\baselineskip}
\begin{theorem}[Two exact ranks for an incidence base]\label{thm:baseexact}
Let $B$ be a nonempty simple bipartite graph without isolated vertices,
with components $B_1,\ldots,B_c$. Put $k_i=|V(B_i)|$,
$m_i=|E(B_i)|$, $k=\sum_i k_i$, $m=\sum_i m_i$, and $A=A_{D_B}$.
For a field $\mathbb K$, let $z_{\mathbb K}(B)$ count the components
for which the integer $m_i-k_i$ becomes zero in $\mathbb K$.
All ranks and dimensions are ordinary nonnegative integers;
only this zero test and the following edge condition are taken
inside the coefficient field. Let $\eta_{\mathbb K}(B)$ count
those components that also satisfy
\[
 (d_B(u)-1)(d_B(v)-1)=1\quad\text{in }\mathbb K
 \quad\text{for every }uv\in E(B_i).
\]
Then
\begin{equation}\label{eq:baseexact}
 \begin{aligned}
 \operatorname{rank}_{\mathbb K}(A-I_m)
   &=k-c-z_{\mathbb K}(B),\\
 \operatorname{rank}_{\mathbb K}[A-I_m\ \ A^{\mathsf T}-I_m]
   &=k-c-\eta_{\mathbb K}(B).
 \end{aligned}
\end{equation}
In characteristic zero, $z_{\mathbb K}(B)$ counts the components
with exactly one cycle, and $\eta_{\mathbb K}(B)$ counts the
components that are cycles. No restriction on four-cycles is needed.
\end{theorem}
\begin{proof}
We work on one connected component, with $m$ edges and $k$ vertices,
and then sum the ranks. Use $C,Q$ from Lemma~\ref{lem:factor} and put
$h=(\mathbf1_L^{\mathsf T},-\mathbf1_R^{\mathsf T})^{\mathsf T}$.
By that lemma,
$A-I=CQC^{\mathsf T}$, $\ker C=\operatorname{span}\{h\}$,
and $C$ has rank $k-1$. In particular,
$\operatorname{col}C^{\mathsf T}=V:=\{y:h^{\mathsf T}y=0\}$.
As $C^{\mathsf T}$ maps onto $V$ and $\ker C=\langle h\rangle$,
we have
\[
 \ker(CQC^{\mathsf T})=(C^{\mathsf T})^{-1}
       (V\cap Q^{-1}\langle h\rangle).
\]
Consequently its rank is $k-1$, except when
$Q^{-1}h\in V$, when it drops to $k-2$. Write
\[
 u=Q^{-1}h=
 \begin{pmatrix}(d_B(x)-1)_{x\in L}\\\mathbf1_R\end{pmatrix},
 \qquad
 v=Q^{-\mathsf T}h=
 \begin{pmatrix}-\mathbf1_L\\(1-d_B(y))_{y\in R}\end{pmatrix}.
\]
Since $h^{\mathsf T}u=h^{\mathsf T}v=m-k$, the exceptional
condition is precisely $m-k=0$ in $\mathbb K$. This proves
the first rank formula on a connected component. If $m-k\ne0$,
the common column space is already the whole endpoint space.

Suppose $m-k=0$ in $\mathbb K$, so $u,v\in V$. The common kernel is
\[
 \ker(A-I)\cap\ker(A^{\mathsf T}-I)
 =(C^{\mathsf T})^{-1}
   \bigl(\operatorname{span}\{u\}\cap\operatorname{span}\{v\}\bigr).
\]
Since $C^{\mathsf T}$ maps onto $V$ and $\dim\ker C^{\mathsf T}=m-k+1$,
the preimage has this dimension plus
$\dim(\operatorname{span}\{u\}\cap\operatorname{span}\{v\})$.
The latter is at most one, since $u$ and $v$ are nonzero
(their right and left blocks, respectively, contain unit entries).
It is one exactly when $u,v$ are proportional. This requires all
left values $d_B(x)-1$ to be one common nonzero scalar and all
right values $d_B(y)-1$ to be its reciprocal. This is equivalent
to the edge condition in the statement: connectivity propagates
the two reciprocal values through the two parts.
Transpose the horizontal concatenation and apply rank--nullity.
Its rank is therefore $k-2$ when that condition holds and $k-1$
otherwise. Summing over the components proves \eqref{eq:baseexact}.

In characteristic zero, a connected graph has $m=k$ exactly when
it has one cycle. The edge condition forces every degree to be
two, so the connected graph is then a cycle. This proves the
last assertions.
\end{proof}

Digraphs may contain directed cycles; \emph{loopless} only excludes
arcs from a vertex to itself. A digraph is strongly connected when every
vertex can reach every other by a directed path; its maximal such
subdigraphs are its strongly connected components. For a loopless digraph $D$, let $F(D)$ be its underlying simple
graph: $uv$ is an edge when at least one of $u\to v$ and $v\to u$
is an arc. For $t\ge1$, the \emph{complete ordered lift} $H_t(D)$
has vertex set $V(D)\times[t]$, where $[t]=\{1,\ldots,t\}$.
Its layers are independent, and $(u,i)$ is adjacent to $(v,j)$ for
$i<j$ exactly when $u\to v$. Let $P_t$ be the strictly upper triangular all-ones matrix of order $t$,
and let $\otimes$ denote the Kronecker product. In layer order,
\[
 A_{H_t(D)}=P_t\otimes A_D+P_t^{\mathsf T}\otimes A_D^{\mathsf T}.
\]
For an incidence base, write $\Gamma(B,t)=H_t(D_B)$.

\Needspace{10\baselineskip}
\begin{lemma}[Rank and triangles in an ordered lift]\label{lem:lift}
Let $D$ be a loopless digraph on $m\ge1$ vertices. Over any field
$\mathbb K$, let $W\subseteq\mathbb K^m$ contain the column spaces of $A_D-I_m$
and $A_D^{\mathsf T}-I_m$, with $\dim W=s$. Then, for every
integer $t\ge1$,
\[
 \rho_{\mathbb K}(H_t(D))\le m+(t-1)s.
\]
For $t\ge3$, the lift is triangle-free if and only if $D$ has
no transitive directed triangle. For $t\le2$, it is always
triangle-free.
\end{lemma}
\begin{proof}
Put $L=A_D-I_m$. Since $P_t+P_t^{\mathsf T}+I_t=J_t$,
\[
 A_{H_t(D)}+I_{mt}
 =J_t\otimes I_m+P_t\otimes L+P_t^{\mathsf T}\otimes L^{\mathsf T}.
\]
The first term has its columns in the diagonal space
$U=\{(y,\ldots,y):y\in\mathbb K^m\}$; the other two have
their columns in $W^t$. As $U\cap W^t$ is the diagonal copy
of $W$, the space $U+W^t$ has dimension $m+st-s$.

A triangle uses three distinct layers. Ordering them increasingly,
its three distinct base vertices form a transitive directed triangle.
Conversely, any such triangle lifts to the first three layers.
\end{proof}

By Lemma~\ref{lem:transitive}, this characterizes all incidence
bases whose complete lifts are triangle-free for every $t$,
including some bases with four-cycles such as $K_{2,r}$.
For incidence bases, Lemma~\ref{lem:factor} gives the bound
$\rho_{\mathbb K}(\Gamma(B,t))\le m+(t-1)(k-c)$ from the introduction.

An independent set can be compressed so that all its intermediate
layers use the same set of base vertices. For incidence bases,
this strengthens the general independence bound.

\begin{proposition}[Endpoint compression and a degree bound]
\label{prop:independence-endpoints}
Let $D$ be a loopless digraph on $m\ge1$ vertices, and let
$a=\alpha(F(D))$. For every integer $t\ge2$,
\begin{equation}\label{eq:independence-endpoints}
 \alpha(H_t(D))=
 \max_{\substack{U,V\subseteq V(D)\\
                   \text{no arc from }U\text{ to }V}}
 \bigl(|U|+|V|+(t-2)|U\cap V|\bigr).
\end{equation}
The sets $U,V$ may intersect. In particular, for all $t\ge1$,
\[\max\{m,ta\}\le\alpha(H_t(D))\le m+(t-1)a.\]
If $D=D_B$ for a simple bipartite graph $B$ of minimum degree
$\delta\ge2$, then
\begin{equation}\label{eq:incidence-independence}
 \alpha(\Gamma(B,t))=m\quad(1\le t\le\delta),\qquad
 \alpha(\Gamma(B,t))\le m+(t-\delta)a\quad(t\ge\delta).
\end{equation}
These statements do not require $B$ to be regular or $C_4$-free.
If $2\le t<\delta$, every maximum independent set contains exactly one
copy of each flag. Its layer map $\ell:E(B)\to[t]$ satisfies
$\ell(e)\ge\ell(f)$ for every arc $e\to f$; conversely, every such
map gives a maximum independent set.
\end{proposition}
\begin{proof}
For an admissible pair $U,V$ in \eqref{eq:independence-endpoints},
put $U$ in the first layer, $V$ in the last, and $U\cap V$ in
all intervening layers. This is independent and has the stated size.
Conversely, let $I$ be any independent set in $H_t(D)$. For each
represented base vertex $v$, let $f_v,\ell_v$ be its first and last
occupied layers and $n_v$ its number of copies. At a cut
$i\in\{1,\ldots,t-1\}$, put
$U_i=\{v:f_v\le i\}$ and $V_i=\{v:\ell_v>i\}$.
There is no arc from $U_i$ to $V_i$, since such an arc would
join the selected copies $(u,f_u)$ and $(v,\ell_v)$.
Their union is the set $S$ of all represented vertices, and
\[
 \frac1{t-1}\sum_{i=1}^{t-1}
 \bigl(|U_i|+|V_i|+(t-2)|U_i\cap V_i|\bigr)
 =|S|+\sum_{v\in S}(\ell_v-f_v)\ge\sum_{v\in S}n_v.
\]
Some admissible pair therefore has value at least $|I|$, proving
\eqref{eq:independence-endpoints}. Its intersection is independent
in $F(D)$, so writing the objective as
$|U\cup V|+(t-1)|U\cap V|$ proves the general upper bound.
A full layer and all copies of an independent base set give the
lower bounds; $t=1$ is immediate.

\Needspace{7\baselineskip}
Now suppose $D=D_B$, where $B=(L\cup R,E)$ has minimum degree
$\delta\ge2$. Take any admissible $U,V\subseteq E$ and let
$\xi,\zeta$ be their indicator vectors. By the endpoint
factorization, $0=\xi^{\mathsf T}A\zeta
=(P_L\xi)^{\mathsf T}(P_R\zeta)$.
These vectors are viewed over the integers, not the coefficient field.
They are nonnegative, so their supports are disjoint; here
$\operatorname{supp}x$ is the set of indices of its nonzero entries.
Let $n_L$ count the left vertices incident with exactly one flag
of $U$, and $n_R$ the right vertices incident with exactly one
flag of $V$. A left vertex incident with one selected flag
contributes $d_B(u)-1$ to $|\operatorname{supp}(P_L\xi)|$;
if it is incident with at least two, it contributes $d_B(u)$.
Consequently
\[
 |\operatorname{supp}(P_L\xi)|\ge |U|+(\delta-2)n_L,
 \qquad
 |\operatorname{supp}(P_R\zeta)|\ge |V|+(\delta-2)n_R.
\]
Every common flag $uv\in U\cap V$ has either exactly one
$U$-flag at $u$ or exactly one $V$-flag at $v$; otherwise both
support vectors are positive at $uv$. Each such singleton endpoint
accounts for at most one common flag, so $|U\cap V|\le n_L+n_R$.
Disjointness of the supports yields the weighted inequality
\[
 |U|+|V|+(\delta-2)|U\cap V|\le m.
\]
For $2\le t\le\delta$, the objective in
\eqref{eq:independence-endpoints} is at most $m$; a full layer
attains it. For $t\ge\delta$, its excess is at most
$(t-\delta)|U\cap V|\le(t-\delta)a$.
This proves \eqref{eq:incidence-independence}.

For the equality assertion, return to an arbitrary independent set $I$
and its first and last layers $f_v,\ell_v$. Put
$L(I)=\sum_{v\in S}(\ell_v-f_v)$, where $S$ is its represented flag set.
Summing the weighted inequality over its $t-1$ cuts gives
$(t-1)|S|+(\delta-1)L(I)\le(t-1)m$. Hence
\[
 |I|\le |S|+L(I)
 \le m-\frac{\delta-t}{t-1}L(I).
\]
If $|I|=m$ and $t<\delta$, then $L(I)=0$ and $|S|=m$.
Thus every flag occurs once. The arc condition is exactly the
condition that these copies form an independent set, proving both
directions. In particular, the layer map is constant on each strongly
connected component of $D_B$.
\end{proof}

The general bound is attained by a transitive tournament: take a full
first layer and repeat the source in every later layer. The affine
family attains the incidence bound.

A directed Hamilton cycle gives spanning cycles in the lifts.
The following interpolation will also let us prescribe their order.

\begin{lemma}[Hamilton cycles in ordered lifts]\label{lem:hamilton-lift}
If a loopless digraph $D$ on $m\ge3$ vertices has a directed
Hamilton cycle, then $H_t(D)$ is Hamiltonian for every $t\ge3$.
If $m$ is odd, then for every $r\ge3$ and $0\le h\le m$,
$H_{r+2}(D)$ has a Hamiltonian induced subgraph of order $rm+2h$
containing $r$ complete layers.
\end{lemma}
\begin{proof}
Retain a directed Hamilton cycle $e_i\to e_{i+1}$, with indices
modulo $m$. Choose a cyclic ordering $\pi_1,\ldots,\pi_t$ of
$[t]$ and put $\pi_{t+1}=\pi_1$. For each $j$, let
$\epsilon_j=1$ if $\pi_j<\pi_{j+1}$ and $-1$ otherwise.
The edges
\[
 (e_i,\pi_j)(e_{i+\epsilon_j},\pi_{j+1})
 \qquad(0\le i<m,\ 1\le j\le t)
\]
form a spanning union of cycles: the preceding and following layers
are distinct, so every vertex has degree two. One layer tour changes
the base index by $\delta=\sum_j\epsilon_j$. The cycles correspond
to the orbits of $i\mapsto i+\delta$ modulo $m$, so their number is
$\gcd(m,\delta)$.

Start from the cyclic order $(1,2,\ldots,t)$. For odd $t$, swap
the disjoint pairs $(3,4),(5,6),\ldots,(t-2,t-1)$, with no swaps
when $t=3$. There are
$(t-1)/2$ descents, including the last step, so $\delta=1$
and the spanning cycle is already Hamiltonian.
For even $t$, swap $(4,5),(6,7),\ldots,(t-2,t-1)$ instead;
the empty list is allowed when $t=4$. Now $\delta=2$, and the
same conclusion holds if $m$ is odd.

If both $m$ and $t$ are even, there are two cycles. The ordering
starts with $1,2,3$ and ends with $t$. Its edges
$(e_0,1)(e_1,t)$ and $(e_0,2)(e_1,3)$ lie in different cycles.
Indeed, if $s_j=\sum_{h<j}\epsilon_h$, then $i-s_j\pmod2$ is
constant on each cycle. It is zero at $(e_0,1)$ and one at $(e_0,2)$.
Replace these edges by
$(e_0,1)(e_1,3)$ and $(e_0,2)(e_1,t)$.
Both are edges of the ordered lift, and the replacement joins the
two cycles into one Hamilton cycle.

For the interpolation, suppose $m$ is odd. The cycle just constructed
on $r$ layers contains all $m$ edges $(e_i,1)(e_{i+1},r)$.
Embed its layers, in order, as $1,4,5,\ldots,r+2$ in
$H_{r+2}(D)$. For any $h$ distinct indices $i$, replace the
edge $(e_i,1)(e_{i+1},r+2)$ by the path
\[
 (e_i,1),\quad(e_{i+1},3),\quad(e_i,2),\quad(e_{i+1},r+2).
\]
Each of its three edges uses the arc $e_i\to e_{i+1}$.
The inserted vertices are pairwise distinct, and none lies in the
original cycle. The replacements therefore give one cycle on
$rm+2h$ vertices. The subgraph induced by these vertices contains
this cycle and all $r$ original layers.
\end{proof}

From the same blocks, we obtain a lower bound in terms of the base rank.

\Needspace{10\baselineskip}
\begin{theorem}[A sharp twofold comparison for ordered block matrices]
\label{thm:rank-comparison}
Let $A$ be an $m\times m$ matrix over a field $\mathbb K$, where
$m\ge1$, and let $K_t(A)$ have diagonal blocks $I_m$, blocks $A$
above the diagonal, and blocks $A^{\mathsf T}$ below the diagonal.
Put $L=A-I_m$, $r=\operatorname{rank}_{\mathbb K}L$,
$s=\operatorname{rank}_{\mathbb K}[L\ L^{\mathsf T}]$, and
$R_t=\operatorname{rank}_{\mathbb K}K_t(A)$. Then $R_1=m$, and
for every $t\ge2$,
\[
 m+(t-2)r\le R_t\le m+(t-1)s\le2R_t.
\]
The factor two cannot be replaced by any smaller universal constant,
even when $A=A_{D_B}$ for a $C_4$-free incidence base $B$.
\end{theorem}
\begin{proof}
For the upper bound, we apply the common-space argument of
Lemma~\ref{lem:lift} with
$W=\operatorname{col}L+\operatorname{col}L^{\mathsf T}$.

Taking the Schur complement of the first $I_m$ gives
$R_t=m+\operatorname{rank}S$, where $S$ is obtained from
$K_{t-1}(A)$ by subtracting $A^{\mathsf T}A$ from every block.
Form the $t-2$ consecutive block-row differences of $S$.
This left multiple has only $-L^{\mathsf T}$ in position $i$ and $L$ in
position $i+1$ of block row $i$.
If $t>2$ and $r>0$, choose a nonsingular minor $L[I,J]$ of order $r$.
Keep rows $I$ in each block row and columns $J$ in block columns
$2,\ldots,t-1$. This gives a block lower triangular minor with
$t-2$ diagonal blocks $L[I,J]$, so
$\operatorname{rank}S\ge(t-2)r$.
For $t=2$ or $r=0$, the lower bound is simply $R_t\ge m$.

Since $s\le m$ and $s\le2r$, for $t\ge2$ we have
\[
 m+(t-1)s=m+s+(t-2)s
 \le2m+2(t-2)r\le2R_t.
\]
For sharpness, take $B=C_{2\ell}$ with $\ell\ge3$ and put $A=A_{D_B}$.
Following the two directions around $B$ shows that $D_B$ is the
disjoint union of two directed $\ell$-cycles. Thus $m=2\ell$,
$r=s=m-2$, and $A^{\mathsf T}A=I_m$ over every field.
The Schur complement gives $R_2=m$, so the common-space ratio
is $(2m-2)/m\to2$. These bases are $C_4$-free.
\end{proof}

The supporting dimension here is the \emph{smallest} common column
space. An arbitrary larger choice need not satisfy the twofold
comparison. Since $r\le m$, we also obtain
$R_t\ge\max\{m,(t-1)r\}$ for $t\ge2$.

For incidence bases, put $e=k-c$. Theorem~\ref{thm:baseexact}
gives $r\ge e/2$: a component on $k_i\ge3$ vertices contributes
at least $k_i-2\ge(k_i-1)/2$, and a single edge contributes one.
Since $e\le m$, the same argument gives
\[
 \frac{m+(t-1)(k-c)}2
 \le\rho_{\mathbb K}(\Gamma(B,t))
 \le m+(t-1)(k-c)
\]
for $t\ge2$, with $t=1$ immediate. For the same even-cycle bases
at $t=2$, the endpoint ratio is $(2m-1)/m\to2$, so this
comparison is sharp as well.

Over $\mathbb R$, ordered layers retain low complement rank even when
direct symmetrization does not. If the base has no opposite arcs and $A=I+L$,
with the columns of $L,L^{\mathsf T}$ in $W$, then
$A+A^{\mathsf T}-2I=L+L^{\mathsf T}$ has rank at most $\dim W$,
but $A+A^{\mathsf T}+I$ acts as $3I$ on $W^\perp$.
By contrast, the lift acts on $(W^\perp)^t$ as
$J_t\otimes I_{W^\perp}$, where $I_{W^\perp}$ is the identity on
$W^\perp$. Its rank is $m-\dim W$, independently of $t$.

\section{Affine parameters and exact incidence ranks}

Fix a prime power $q\ge2$. Let $B_q$ be the incidence graph between
the points $(x,y)\in\F_q^2$ and the nonvertical lines
$\ell_{a,b}:y=ax+b$, where $a,b\in\F_q$. Every point and every
line has degree $q$, so $B_q$ has $2q^2$ vertices and $q^3$ edges.
Write $G(q,t)=\Gamma(B_q,t)$.

\begin{lemma}[The affine base and its lift]
The graph $B_q$ is connected and $C_4$-free. Its endpoint space has
dimension $2q^2-1$ over every field. The graph $G(q,t)$ is
triangle-free, has $q^3t$ vertices, and is $(q-1)^2(t-1)$-regular.
In particular, $|E(G(q,t))|=q^3t(q-1)^2(t-1)/2$.
Moreover, $\alpha(F(D_{B_q}))=q^2$ and $\chi(F(D_{B_q}))=q$.
The lift has the independence and chromatic numbers stated in
Theorem~\ref{thm:affine}; thus $\alpha(G(q,t))\chi(G(q,t))=q^3t$.
\end{lemma}
\begin{proof}
Two points with different first coordinates lie on exactly one
nonvertical line; two distinct points with the same first coordinate
lie on none. This excludes four-cycles. Joining the latter two points
through a point with a different first coordinate proves connectivity.
The dimension and triangle-freeness follow from
Lemmas~\ref{lem:factor} and~\ref{lem:lift}.
From a flag, an outgoing arc chooses a different line through its
point and then a different point on that line. Thus the directed
base has in-degree and out-degree $(q-1)^2$. A lifted vertex has
this many neighbors in each other layer, proving the degree and
edge-count formulas.
An arc joins flags whose lines have different slopes: parallel
lines cannot both pass through the first flag's point. Thus coloring
by line slope uses $q$ colors, each on exactly $tq^2$ vertices.
Coloring by layers uses $t$ colors, each on $q^3$ vertices.
These colorings give the stated lift bounds and
$\alpha(F(D_{B_q}))\ge q^2$.

For the reverse base bound, $F(D_{B_q})$ is the superline graph
of $B_q$: two disjoint incidence edges are adjacent when an edge
of $B_q$ joins their endpoints. By Kostochka et al.\ \cite[Lemma~2.1]{KPR}, its independence number is at most
$\alpha(B_q)$. Since $B_q$ is regular and bipartite, Hall's theorem~\cite{Hall}
gives a perfect matching, so $\alpha(B_q)=q^2$. The slope coloring
therefore attains both base bounds. Apply
\eqref{eq:incidence-independence} with $m=q^3$, $\delta=q$
and $a=q^2$. It gives $\alpha(G(q,t))\le q^3$ for $t\le q$
and $\alpha(G(q,t))\le tq^2$ for $t\ge q$.
The full layer and slope class attain these bounds. Finally,
$\chi\ge |V|/\alpha=\min\{q,t\}$ matches the two colorings.
This proves the independence and coloring assertions in
Theorem~\ref{thm:affine}, including $q=2$ and $t=1$.
\end{proof}

The endpoint-space bound in Lemma~\ref{lem:lift} proves the
field-uniform upper bound in Theorem~\ref{thm:affine}. Before proving
real equality, we identify the smaller common space that occurs
in exceptional characteristics.

\begin{corollary}[The affine common space in every characteristic]
\label{cor:affine-common-field}
For a field $\mathbb K$, let $\varepsilon_{\mathbb K}(q)$ be one
if $q(q-2)=0$ in $\mathbb K$, and zero otherwise. Put
$s_{\mathbb K}(q)=2q^2-1-\varepsilon_{\mathbb K}(q)$ and
$A=A_{D_{B_q}}$. Then
\[
 \operatorname{rank}_{\mathbb K}(A-I)
 =\operatorname{rank}_{\mathbb K}[A-I\ \ A^{\mathsf T}-I]
 =s_{\mathbb K}(q),
\]
and, for every $t\ge2$,
\[
 q^3+(t-2)s_{\mathbb K}(q)
 \le \rho_{\mathbb K}(G(q,t))
 \le q^3+(t-1)s_{\mathbb K}(q).
\]
The gap is at most $2q^2-1$, independently of $t$ and $\mathbb K$.
For $t=1$, the rank is $q^3$.
\end{corollary}
\begin{proof}
The shifted-rank exception in Theorem~\ref{thm:baseexact} is
$q^2(q-2)=0$, equivalently $q(q-2)=0$. When it holds,
$(q-1)^2=1$, so the common-space dimension also decreases by one.
Apply Theorems~\ref{thm:baseexact} and~\ref{thm:rank-comparison}.
\end{proof}

\subsection*{Singular pairs and exact affine ranks}

All ranks and inner products in the rest of this section are over
$\mathbb R$, unless another coefficient field is specified.
Over $\mathbb R$ or $\mathbb C$, vector norms are Euclidean
and matrix norms are the induced operator norms.
A singular-pair recurrence determines the rank at every layer count;
the affine formulas and the three-layer classification follow from it.

\begin{theorem}[Spectral formula for biregular incidence lifts]
\label{lem:singular-blocks}
Let $B=(L\cup R,E)$ be a nonempty simple bipartite graph of degree
$a\ge2$ on $L$ and degree $b\ge2$ on $R$, with $a+b>4$.
Write $m=|E|$, $k=|L|+|R|$, and let $c$ be its number of
components and $N$ its bipartite incidence matrix.
Write $\mathbf1_{\{P\}}$ for the indicator of a condition $P$, and put
\[
 \zeta_B=(|L|-\rk N)\mathbf1_{\{a=2\}}
        +(|R|-\rk N)\mathbf1_{\{b=2\}}.
\]
Define the polynomials
\begin{equation}\label{eq:g-recurrence}
 g_{t+2}(x)=(2-x)g_{t+1}(x)-g_t(x)+2(a-1)(b-1),
 \qquad g_1(x)=1,\quad g_2(x)=(a-2)(b-2).
\end{equation}
Let $\kappa_t(B)$ count, with multiplicity, the eigenvalues
$0<\lambda<ab$ of $NN^{\mathsf T}$ for which $g_t(\lambda)=0$.
Then $0\le\kappa_t(B)\le\rk N-c$, and for every $t\ge1$,
\begin{equation}\label{eq:biregular-rank}
 \rho(\Gamma(B,t))
 =m+(t-1)(k-c)-\kappa_t(B)-\zeta_B\mathbf1_{\{t=2\}}.
\end{equation}
\end{theorem}
\begin{proof}
The block identities below hold for all $a,b\ge2$, including
$a=b=2$. Put $A=A_{D_B}$, $W=\operatorname{col}[X\ Y]$,
and $d=(a-1)(b-1)$. By Lemma~\ref{lem:factor}, both $A$ and
$A^{\mathsf T}$ act as the identity on $W^\perp$.
For a vector $v$ on $R$, Cauchy--Schwarz gives
\[
 \|Nv\|^2=\sum_{u\in L}\Bigl(\sum_{w\sim u}v_w\Bigr)^2
 \le a\sum_{u\in L}\sum_{w\sim u}|v_w|^2=ab\|v\|^2.
\]
Equality requires $v$ to be constant on the neighbors of every
left vertex; because every vertex has positive degree, this means
one constant on the right part of each connected component.
Thus $\sqrt{ab}$ is simple on each component and contributes
one constant endpoint direction, with scalar action $d$.
The left and right zero singular spaces give scalar actions $1-a$
and $1-b$, with dimensions $|L|-\rk N$ and $|R|-\rk N$,
respectively.

Choose orthonormal singular pairs $Nv=zu$, $N^{\mathsf T}u=zv$
for the remaining values $0<z<\sqrt{ab}$. The spaces spanned by
$Xu,Yv$ are mutually orthogonal, and the endpoint factorization
gives their Gram matrix and shifted coordinate matrices:
\[
 M_z=\begin{pmatrix}a&z\\z&b\end{pmatrix},\qquad
 Q_z=\begin{pmatrix}-1&z\\0&-1\end{pmatrix},\qquad
 (A-I)|_z=Q_zM_z,\quad
 (A^{\mathsf T}-I)|_z=Q_z^{\mathsf T}M_z.
\]
The scalar and mixed dimensions add to
$c+(|L|-\rk N)+(|R|-\rk N)+2(\rk N-c)=k-c=\dim W$.
Thus the decomposition exhausts $W$, including when $|L|\ne|R|$
and when the two zero singular spaces have different dimensions.
A scalar direction with value $\lambda$ gives the layer matrix
$(1-\lambda)I_t+\lambda J_t$.

Since $M_z$ is positive definite and $Q_z$ is invertible,
consecutive kernel equations in the $t$-layer block give
$Q_zM_zx_{i+1}=Q_z^{\mathsf T}M_zx_i$, hence
$x_i=M_z^{-1}R_z^{i-1}M_zx_1$, where
\[
 R_z=Q_z^{-1}Q_z^{\mathsf T}
     =\begin{pmatrix}1-z^2&z\\-z&1\end{pmatrix},\qquad
 D_z=I_2+M_zQ_z.
\]
Substitution in the first equation, with $y_1=M_zx_1$, gives
\[
 F_ty_1=0,\qquad F_t=I_2+D_z\sum_{j=1}^{t-1}R_z^j.
\]
Conversely, this equation and the consecutive differences imply
all block equations. Thus the mixed block and $F_t$ have the same
nullity, also when $t=1$ and the sum is empty.

Put $E_z=R_z-I_2-D_zR_z$. Direct multiplication gives
\[
 D_zR_z=\begin{pmatrix}1-a&0\\(b-2)z&1-b\end{pmatrix},
 \quad \det D_z=\det E_z=d,\quad
 \det R_z=1,\quad \det(R_z-I_2)=z^2.
\]
Writing $\operatorname{tr}$ for trace and $\operatorname{adj}$ for
the classical adjugate, take determinants in
$F_t(R_z-I_2)=D_zR_z^t+E_z$ to obtain
\[
 z^2\det F_t=2d+
 \operatorname{tr}\bigl(\operatorname{adj}(E_z)D_zR_z^t\bigr).
\]
Cayley--Hamilton, with $\operatorname{tr}R_z=2-z^2$, gives
$\det F_t=g_t(z^2)$ through \eqref{eq:g-recurrence};
$F_1=I_2$ and $F_2=I_2+D_zR_z$ give its initial values.

Now assume $a+b>4$, so $d>1$. The scalar layer eigenvalues are
$1-\lambda$ with multiplicity $t-1$ and $1+(t-1)\lambda$.
For $\lambda\in\{d,1-a,1-b\}$, singularity occurs only when
$\lambda=-1$ and $t=2$, with nullity one. The total scalar
loss is therefore $\zeta_B\mathbf1_{\{t=2\}}$.
Also $F_t\ne0$: otherwise
$D_z^{-1}=-\sum_{j=1}^{t-1}R_z^j$ would force $D_z$ to commute
with $R_z$, but
\[
 (D_zR_z-R_zD_z)_{11}=-(a-2)z^2,\qquad
 (D_zR_z-R_zD_z)_{21}=(b-a)z
\]
cannot both vanish: if $a\ne2$ the first is nonzero, while
if $a=2$, the assumption $a+b>4$ makes the second nonzero.
Each mixed block therefore loses one rank, never two, exactly
when $g_t(z^2)=0$. The endpoint blocks contribute
$t(k-c)-\kappa_t(B)-\zeta_B\mathbf1_{\{t=2\}}$, and
$(W^\perp)^t$ contributes $m-k+c$ through
$J_t\otimes I_{W^\perp}$. This proves \eqref{eq:biregular-rank};
there are $\rk N-c$ positive nonprincipal singular pairs.
\end{proof}

The affine base is the Wenger graph $W_1(q)$: the equation
$y=ax+b$ becomes $p_2+l_2=p_1l_1$ under $p=(x,y)$ and $l=(a,-b)$.
By Cioab\u{a} et al.\ \cite[Theorem~2.1]{CLL}, its point--line
incidence matrix has singular values $q$ once, $\sqrt q$ with
multiplicity $q(q-1)$, and $0$ with multiplicity $q-1$.

\begin{proof}[Real equality in Theorem~\ref{thm:affine}]
For $B_q$, the only positive nonprincipal incidence singular value
is $\sqrt q$. By \eqref{eq:biregular-rank}, it suffices to show
$g_t(q)\ne0$ for every $t\ge1$, with $a=b=q$ in
\eqref{eq:g-recurrence}. In fact, this holds for every integer $q\ge3$.

Write $g_t=g_t(q)$. For $q=3$, the values repeat as $1,1,6$.
For $q=4$, induction in \eqref{eq:g-recurrence} gives
$g_t=4t-3$ for odd $t$ and $g_t=12-4t$ for even $t$.

Suppose $q\ge5$. Modulo $q-2$, the recurrence gives
$g_{2j+1}\equiv1$, so no odd-indexed term vanishes.
Put $c_q=2(q-1)^2$ and $b_t=(-1)^t(qg_t-c_q)$; then
$b_{t+2}=(q-2)b_{t+1}-b_t$.
For $q\ge6$, we have $b_1=2(q-1)^2-q>0$ and
$b_2-b_1=(q-6)(q-1)^2+2>0$, so all $b_t$ are positive
and increasing. For $q=5$, the first four values are
$(27,13,12,23)$, and the same induction applies from the third
term onward. Hence $qg_t=c_q+b_t>0$ for every even $t$.

Thus $\kappa_t(B_q)=0$. Since $B_q$ is connected with $q^3$
edges and $2q^2$ vertices, \eqref{eq:biregular-rank} gives
$\rho(G(q,t))=q^3+(2q^2-1)(t-1)$.
\end{proof}

The same algebraic blocks give exact ranks in certain coefficient fields.
The field congruences below concern the hypotheses; all dimensions
and rank formulas are integer-valued. Real singular-value arguments
are not used in positive characteristic.

Write $\operatorname{diag}(M_1,\ldots,M_s)$ for a block-diagonal matrix
(and likewise for scalar diagonal entries).
Write $\operatorname{char}\mathbb K$ for the characteristic of a
field $\mathbb K$, and $\operatorname{PG}(2,q)$ for the projective
plane whose points and lines are the one- and two-dimensional
subspaces of $\F_q^3$, with incidence given by containment.

\begin{proposition}[Periodic ranks in selected characteristics]
\label{prop:binary-rank}
Let $q$ be a prime power, $t\ge1$ an integer, and $\mathbb K$ a field.
If $\operatorname{char}\mathbb K\ne2$ and $q=2$ in $\mathbb K$, put
$\delta_{\mathbb K}(t)=\mathbf1_{\{t-2=0\text{ in }\mathbb K\}}$.
Then
\[
 \rho_{\mathbb K}(G(q,t))
 =q^3+(2q^2-2)(t-1)
 -2(q-1)\bigl(q\mathbf1_{\{t\equiv2\pmod4\}}
                     +\delta_{\mathbb K}(t)\bigr).
\]
Over $\mathbb R$, necessarily $q=2$, giving
$\rho(G(2,t))=6t+2-4\mathbf1_{\{t\equiv2\pmod4\}}
-2\mathbf1_{\{t=2\}}$.

If $\operatorname{char}\mathbb K=p\ge5$ and $q\equiv1\pmod p$,
let $B$ be the point--line incidence graph of $\operatorname{PG}(2,q)$.
Put $v=q^2+q+1$ and $m=(q+1)v$. Then
\begin{equation}\label{eq:projective-rank}
 \rho_{\mathbb K}(\Gamma(B,t))
 =m+(2v-2)\bigl(t-1-\mathbf1_{\{t\equiv2\pmod6\}}\bigr).
\end{equation}
\end{proposition}
\begin{proof}
Let $B$ denote the relevant incidence base and put $m=|E(B)|$.
Work over an algebraic closure, which preserves rank.
The affine incidence matrix, with points grouped by first coordinate,
and the projective incidence matrix satisfy, respectively,
\[
 NN^{\mathsf T}=qI_{q^2}+J_{q^2}-I_q\otimes J_q,
 \qquad NN^{\mathsf T}=qI_v+J_v.
\]
In the affine case the constant vectors, the vectors summing to zero
within each group, and the vectors constant on each group with total
sum zero form eigenspaces with eigenvalues $q^2,q,0$ and dimensions
$1,q(q-1),q-1$, respectively. Each line meets each group once,
so the last space is $\ker N^{\mathsf T}$; the same holds on
the line side. In the projective case the eigenvalues are
$(q+1)^2,q$, with multiplicities $1,v-1$.
Under the respective field hypotheses, the affine eigenvalues are
$4,2,0$ and the projective eigenvalues are $4,1$, all distinct.
The displayed eigenspaces are therefore mutually orthogonal and
nondegenerate for the standard bilinear form. Over the algebraic
closure, choose an orthonormal basis in each nonzero eigenspace.
For an eigenvector $u$ with eigenvalue $z^2\ne0$, put
$v=z^{-1}N^{\mathsf T}u$; then $v^{\mathsf T}v=u^{\mathsf T}u$,
$Nv=zu$, and $N^{\mathsf T}u=zv$.
The mixed endpoint Gram determinants are $q(q-1)=2$ and
$(q+1)^2-q=3$, respectively, so these blocks are nonsingular.
Moreover $m=8$ or $6$ in the coefficient field, respectively;
thus the constant direction is nondegenerate as well.
The algebraic calculation in Theorem~\ref{lem:singular-blocks}
now applies without using real positivity.

In both cases the principal direction has scalar action one,
so it and $W^\perp$ contribute $m-|V(B)|+2$ to the rank.
In the affine case each of the $2(q-1)$ zero singular directions
gives $2I_t-J_t$, of rank $t-\delta_{\mathbb K}(t)$.
On each of its $q(q-1)$ mixed pairs, $D_z=R_z$ and $R_z^2=-I_2$.
Thus $F_t=I_2,0,-R_z,I_2-R_z$ for
$t\equiv1,2,3,0\pmod4$, respectively; only the zero matrix is
singular because $\det(I_2-R_z)=2$.
Each pair contributes $2t-2\mathbf1_{\{t\equiv2\pmod4\}}$,
proving the first formula.

In the projective case there are no zero singular directions.
Each of the $v-1$ mixed pairs has
\[
 R_z=\begin{pmatrix}0&1\\-1&1\end{pmatrix},\qquad
 D_z=R_z^2=R_z-I_2,\qquad R_z^3=-I_2,
 \qquad F_t=R_z^t-R_z^2.
\]
This matrix is zero when $t\equiv2\pmod6$; in every other
class its determinant is $1$, $3$, or $4$, hence nonzero.
Adding the mixed contributions to $m-2v+2$ proves
\eqref{eq:projective-rank}, including $t=1,2$.
\end{proof}

In the affine case of Proposition~\ref{prop:binary-rank}, with odd
characteristic $p$, $q\equiv2\pmod p$, and $t\ge2$, both
bounds in Corollary~\ref{cor:affine-common-field} are attained:
the lower bound when $t\equiv2\pmod{4p}$, and the upper bound
when $t\equiv1\pmod{4p}$. Thus both endpoints are attained
infinitely often in $t$.

For instance, the $625$-vertex graph $G(5,5)$ has rank $321$
over $\mathbb R$ but $309$ over $\mathbb F_3$.
Over $\mathbb R$, the exact rank also shows that $-1$ has
adjacency multiplicity $(t-1)(q^3-2q^2+1)$ for $q\ge3$;
when $t\asymp q$, this is $N-O(N^{3/4})$ for $N=q^3t$.

\section{Sharp bounds and stability}

For a nonempty $C_4$-free bipartite graph without isolated vertices,
write $m$ for its number of edges, $k$ for its number of vertices,
and $a,b$ for its part sizes. Reiman~\cite{Reiman} and Keevash et al.\ \cite[Proposition~1.4]{KSV} give the classical bounds
\begin{equation}\label{eq:reiman-total}
 m\le\frac{k}{4}\bigl(1+\sqrt{2k-3}\bigr)
   =\left(\frac{k}{2}\right)^{3/2}+O(k),
\end{equation}
and the rectangular inequalities $m^2\le ba(a-1)+bm$ and
$m\le a\sqrt b+b$. The symmetric bound is the $\ell=2$ case
of the latter result (with $k=2$ immediate).

\Needspace{12\baselineskip}
\begin{lemma}[Edge bounds for admissible incidence bases]
\label{lem:rank-edges}
Let $B$ be a nonempty simple bipartite graph such that $D_B$ has no
transitive directed triangle. Write $m=|E(B)|$, let $k$ count
the non-isolated vertices, and put
$r=\operatorname{rank}_{\mathbb K}(A_{D_B}-I_m)$.
Set $\Phi(x)=x(1+\sqrt{2x-3})/4$ for $x\ge2$.
For all sufficiently large $k$ and $r$, respectively,
\[
 m\le\Phi(k),\qquad m\le\Phi(r+2).
\]
Consequently, uniformly in $B$ and $\mathbb K$,
\begin{equation}\label{eq:rank-edge}
 \begin{aligned}
 m&\le(r/2)^{3/2}+\frac r4+O(\sqrt r),\qquad
 r\ge2m^{2/3}-\frac23m^{1/3}-\frac43-O(m^{-1/3}).
 \end{aligned}
\end{equation}
\end{lemma}
\begin{proof}
Delete all vertices of original degree at most two and any
resulting isolates. By
Lemma~\ref{lem:transitive}, the core is $C_4$-free; if it has
$k'$ vertices, at most $2(k-k')$ edges were deleted.
For $k\ge26$, the function $\Phi(x)/x$ is increasing and
$\Phi(k)/k\ge2$. Thus \eqref{eq:reiman-total} gives
$m\le\Phi(k')+2(k-k')\le\Phi(k)$, with the empty core
handled by $m\le2k$.

For each nonempty component, use subscripts $i$ for its counts
and shifted rank. Theorem~\ref{thm:baseexact} gives
$r_i\ge1$ and $k_i\le r_i+2$.
The function $h(x)=\Phi(x+2)/x$ is increasing for $x\ge4$:
the sign of its derivative is the sign of
$x^2-2x-2-2\sqrt{2x+1}$, which is nonnegative there.
If $k_i\ge26$, it follows that
$m_i\le\Phi(r_i+2)\le r_i h(r)$.
If $k_i\le25$, the elementary bound
$m_i\le\lfloor k_i^2/4\rfloor$ instead gives $m_i\le7r_i$.
For $r\ge400$ we have $h(r)>7$, so summation yields
$m\le rh(r)=\Phi(r+2)$.
Expansion gives the first estimate in \eqref{eq:rank-edge}.
For the inverse, write $m=(x+1)(x^2+x+1)$ with $x\ge0$.
Since $\Phi$ is increasing and $\Phi(2x^2+2x+2)=m$,
the bound $m\le\Phi(r+2)$ gives $r\ge2x^2+2x$ for large $m$.
With $w=m^{1/3}$, the defining cubic gives
$x=w-2/3-2/(9w)+O(w^{-2})$.
Substitution proves the second estimate in \eqref{eq:rank-edge}.
Bounded orders are absorbed in the absolute error constants.
\end{proof}

\begin{proof}[Proof of Theorem~\ref{thm:methodconstant}]
Put $m=|E(B)|$, $r=\operatorname{rank}_{\mathbb K}(A_{D_B}-I)$,
and $R=\rho_{\mathbb K}(\Gamma(B,t))$.
We may assume $R\le(C_*+1)N^{3/4}$. For large $N$ this forces
$t\ge3$, since a full independent layer gives $R\ge m\ge N/2$
when $t\le2$. The base therefore satisfies
Lemma~\ref{lem:rank-edges}, which gives $r\ge c_0m^{2/3}$
for an absolute $c_0>0$. Since $t=N/m$ and $t\ge3$, the bound $R\ge(t-1)r$ gives
$R\ge c_0N/(2m^{1/3})$. Since
$R\le(C_*+1)N^{3/4}$, we get
$m\ge c_1N^{3/4}$; on the other hand
$m\le R\le(C_*+1)N^{3/4}$.
Thus $m\asymp N^{3/4}$ and $t=N/m\asymp N^{1/4}$.
In particular $m\to\infty$, so the large-$m$ inverse estimate
applies uniformly.

Put $y=m/N^{3/4}$ and $\epsilon=N^{-1/4}$. Substituting the
more precise inverse estimate \eqref{eq:rank-edge} in
$R\ge m+(t-2)r$ gives, uniformly on the resulting compact
interval for $y$,
\[
 R\ge N^{3/4}\bigl(f(y)-\epsilon g(y)+\epsilon^2h(y)\bigr)-O(1),
\]
where
\[
 f(y)=y+2y^{-1/3},\quad
 g(y)=4y^{2/3}+\frac23y^{-2/3},\quad
 h(y)=\frac43(y^{1/3}-y^{-1}).
\]
The error is bounded because the $O(m^{-1/3})$ term in the
inverse estimate is multiplied by $t-2=O(N^{1/4})$.
The function $f$ has its unique nondegenerate minimum at
$a_*=(2/3)^{3/4}$, with $f(a_*)=C_*$ and $g(a_*)=D_*$.
On the fixed compact interval, the perturbed expression converges
uniformly to $f$, so every minimizer lies near $a_*$.
There $f''$ is bounded away from zero. The derivative equation
first gives $y-a_*=O(\epsilon)$ and then
$y=a_*+\epsilon g'(a_*)/f''(a_*)+O(\epsilon^2)$.
Taylor expansion yields
\[
 \min_y\{f(y)-\epsilon g(y)+\epsilon^2h(y)\}
 =C_*-D_*\epsilon+
 \left(h(a_*)-\frac{g'(a_*)^2}{2f''(a_*)}\right)\epsilon^2
 +O(\epsilon^3).
\]
Here the minimum may be taken on any fixed compact interval
containing all the relevant $y$ and $a_*$ in its interior.
Direct evaluation gives
$h(a_*)-g'(a_*)^2/(2f''(a_*))=-\frac{13}{6}(2/3)^{1/4}=-E_+$.
Since $N^{3/4}\epsilon^3=1$, this proves the bound; an absolute
constant covers the remaining orders.

For sharpness, let $B$ be the incidence graph of
$\operatorname{PG}(2,q)$, where $q=11^{2j+1}$, set
$t=(3q-5)/2$, and take $\mathbb K=\mathbb F_5$.
Then $q\equiv1\pmod5$, $q\equiv3\pmod4$, and $t\equiv2\pmod6$.
With $N=(q^3+2q^2+2q+1)t$, \eqref{eq:projective-rank} gives
\[
 \begin{aligned}
 \rho_{\mathbb F_5}(\Gamma(B,t))
 &=4q^3-4q^2-7q+1\\
 &=C_*N^{3/4}-D_*N^{1/2}
   -\frac{17}{8}(2/3)^{1/4}N^{1/4}+O(1).
 \end{aligned}
\]
This example proves the stated sharpness of $C_*$ and $D_*$.
It also gives the displayed third-order correction for this sequence.
\end{proof}

Near equality also forces geometric regularity.

\begin{corollary}[Stability and a geometric criterion for optimality]

Let $H_j=\Gamma(B_j,t_j)$ be triangle-free complete ordered
incidence lifts of nonempty simple bipartite graphs $B_j$, with
$N_j=|V(H_j)|\to\infty$. After discarding isolated
vertices, write $B_j=(L_j\cup R_j,E_j)$, $m_j=|E_j|$, and
$k_j=|L_j|+|R_j|$. Put $a_*=(2/3)^{3/4}$.
The asymptotic relation
\[
 \rho_{\mathbb K_j}(H_j)=(C_*+o(1))N_j^{3/4}
\]
for some sequence of fields $\mathbb K_j$ is equivalent to
\[
 m_j=(a_*+o(1))N_j^{3/4},\qquad
 k_j=(2+o(1))m_j^{2/3}.
\]
Under these equivalent conditions, the rank asymptotic is uniform
over every coefficient field. One component contains all but
$o(m_j^{2/3})$ vertices and
$o(m_j)$ edges. Both parts of $B_j$, and of that component,
have $(1+o(1))m_j^{2/3}$ vertices.
Deleting the vertices of original degree at most two removes
only $o(m_j^{2/3})$ vertices and edges and leaves a $C_4$-free
induced graph.

Put $a_j=|L_j|$ and $b_j=|R_j|$, and let $d(v)$ be the degree
of $v$ in $B_j$ before this deletion. Then
\[
 \sum_{u\in L_j}\left(d(u)-\frac{m_j}{a_j}\right)^2
 +\sum_{v\in R_j}\left(d(v)-\frac{m_j}{b_j}\right)^2
 =o(m_j^{4/3}).
\]
For every fixed $\varepsilon>0$, all but $o(m_j^{2/3})$
vertices have degree between $(1-\varepsilon)m_j^{1/3}$ and
$(1+\varepsilon)m_j^{1/3}$. All but $o(m_j^{4/3})$ pairs
in either part have exactly one common neighbor. Let $e_{B_j}(S,T)$
count the edges between $S$ and $T$. Then
\[
 \max_{S\subseteq L_j,\,T\subseteq R_j}
 \left|e_{B_j}(S,T)-\frac{m_j|S||T|}{a_jb_j}\right|
 =o(m_j).
\]

For a quantitative version, omit the index $j$. Suppose
$\rho_{\mathbb K}(H)\le(C_*+\varepsilon)N^{3/4}$ for one field,
where $0\le\varepsilon\le1$, and put $\eta=\varepsilon+N^{-1/4}$.
For any field $\mathbb L$, write
$r_{\mathbb L}=\operatorname{rank}_{\mathbb L}(A_{D_B}-I)$. Then
\[
 \begin{aligned}
 m&=a_*N^{3/4}\bigl(1+O(\sqrt\eta)\bigr),&
 t&=a_*^{-1}N^{1/4}\bigl(1+O(\sqrt\eta)\bigr),\\
 k&=2m^{2/3}\bigl(1+O(\eta)\bigr),&
 r_{\mathbb L}&=2m^{2/3}\bigl(1+O(\eta)\bigr).
 \end{aligned}
\]
All implied constants are absolute and uniform in the fields.
\end{corollary}
\begin{proof}
We first prove the quantitative parameter estimates, then recover
the global structure from the $C_4$-free core.
Assume the quantitative rank hypothesis and put
$r=\operatorname{rank}_{\mathbb K}(A_{D_B}-I)$.
The proof of Theorem~\ref{thm:methodconstant} gives
$m\asymp N^{3/4}$ and $t\asymp N^{1/4}$, with absolute constants;
bounded orders are absorbed by enlarging the error constants.
Set $y=m/N^{3/4}$ and $u=r/(2m^{2/3})-1$.
For large $N$, Lemma~\ref{lem:rank-edges} gives
$u\ge-O(N^{-1/4})$, while $r=O(m^{2/3})$ and
$\rho_{\mathbb K}(H)\ge m+(t-2)r$ imply
\[
 C_*+\varepsilon
 \ge f(y)+2y^{-1/3}u-O(N^{-1/4}),
 \qquad f(y)=y+2y^{-1/3}.
\]
Since $f\ge C_*$, this gives $u=O(\eta)$ and
$f(y)-C_*=O(\eta)$. The unique minimizer $a_*$ is
nondegenerate, and $y$ stays in a compact subinterval of
$(0,\infty)$; hence $y=a_*+O(\sqrt\eta)$.
This proves the assertions about $m,t,r$.

For the component ranks $r_i$, the componentwise edge bound gives
\[
 \sum_i(r_i/r)^{3/2}\ge1-O(\eta),\qquad
 \sum_i(r_i/r)^{3/2}\le\sqrt{\max_i(r_i/r)}.
\]
Consequently $r'=r-\max_i r_i=O(\eta r)$.
Since $r_i\ge1$ and $k_i\le3r_i$, all other components
have $O(\eta m^{2/3})$ vertices and
$O(\eta^{3/2}m)$ edges. Also $c\le1+r'$, so the exact
base formula gives $k=r+O(c)=2m^{2/3}(1+O(\eta))$
and the same shifted-rank estimate over every field.
The incidence upper bound now yields
$\rho_{\mathbb L}(H)\le(C_*+O(\eta))N^{3/4}$ uniformly in
$\mathbb L$, and Theorem~\ref{thm:methodconstant} supplies
the matching lower estimate.

Along the given sequence take
$\varepsilon_j=\max\{0,\rho_{\mathbb K_j}(H_j)/N_j^{3/4}-C_*\}=o(1)$.
The quantitative hypothesis is then satisfied with nonnegative
$\varepsilon_j$, proving the forward implication and the
dominant-component assertion.
Conversely, the geometric conditions give
$\rho_{\mathbb K}(H_j)\le m_j+t_jk_j
=(C_*+o(1))N_j^{3/4}$ uniformly in $\mathbb K$.
Together with Theorem~\ref{thm:methodconstant}, this proves
the equivalence and uniformity.

Now delete from the entire base the vertices of original degree
at most two and any resulting isolated vertices. By
Lemma~\ref{lem:transitive}, the remaining graph $C_j$ is $C_4$-free
and has $k_j'$ vertices and
$m_j'\ge m_j-2k_j\sim m_j$ edges. By
\eqref{eq:reiman-total}, $k_j'\ge(2-o(1))m_j^{2/3}$.
Therefore $k_j-k_j'=o(k_j)$, and deleting the original low-degree
vertices removes only $o(k_j)$ edges.

The core has average degree $2m_j'/k_j'\sim\sqrt{k_j'/2}$.
By Keevash et al.\ \cite[Theorem~5.1]{KSV}, with the cycle-length
parameter $\ell=2$,
\[
 e_{C_j}(S',T')=
 \frac{4m_j'}{(k_j')^2}|S'||T'|+o(m_j)
\]
uniformly over subsets of its two parts. Taking the full parts,
of sizes $a',b'$, gives $4a'b'/(k_j')^2\to1$. Thus
$a',b'\sim k_j'/2\sim m_j^{2/3}$, and restoring $o(k_j)$
vertices gives $a_j,b_j\sim m_j^{2/3}$, also for the dominant
component. The displayed density differs from $m_j/(a_jb_j)$
by $o(m_j/k_j^2)$. Restoring vertices changes every subset-size
product by $o(k_j^2)$, and restoring edges changes every cut
by $o(k_j)$, proving the uniform cut estimate in $B_j$.

Suppress $j$, and put $k=k_j$, $m=m_j$, $a=a_j$, $b=b_j$.
Cauchy--Schwarz and $C_4$-freeness give
\[
 \frac12\left(\frac{(m_j')^2}{a'}-m_j'\right)
 \le\sum_{u\in L(C_j)}\binom{d_{C_j}(u)}2
 \le\binom{b'}2.
\]
Both endpoints are $(1/2+o(1))m^{4/3}$. Thus all but
$o(m^{4/3})$ pairs in either part of $C_j$ have exactly one
common neighbor. Each restored edge belongs to at most $2k$
paths of length two; restoring $o(k)$ edges changes only
$o(k^2)$ such paths. Only $o(k^2)$ pairs touch a restored vertex,
which proves the common-neighbor assertion in $B_j$. Also,
$\sum_{u\in L_j}\binom{d(u)}2\le\binom b2+o(k^2)$, so
\[
 0\le\sum_{u\in L_j}\left(d(u)-\frac ma\right)^2
 \le b(b-1)+m-\frac{m^2}{a}+o(k^2)=o(m^{4/3}).
\]
The same argument on $R_j$ proves the variance assertion;
Chebyshev's inequality then gives the concentration of degrees.
\end{proof}

\section{Graph parameters and applications}

Recall that our definition of $\msr(G)$ allows edge entries to
vanish. The faithful convention additionally requires nonzero
entries at every edge.
Rosenfeld's almost-orthogonality bound, recalled by
Balla et al.\ \cite[Section~1.2]{BallaLetzterSudakov}, gives
$\msr(G)\ge\lceil |V(G)|/2\rceil$ for triangle-free $G$.
The same bound applies to the weak convention used here, including
Hermitian Gram matrices: the nonzero off-diagonal entries of an
admissible Gram matrix define a spanning subgraph of $G$, which is
still triangle-free. Apply Deaett's bound~\cite{Deaett} on each
nontrivial component of this support graph; isolated vertices
contribute one-dimensional diagonal blocks.

A matching is a set of edges with no common endpoints; it is perfect
if it covers every vertex. Let $\mu(F)$ be the largest size of a matching.
Combining the classical semidefinite-rank bound with the
Tutte--Berge formula~\cite{Berge} gives the following consequence.

\begin{corollary}[A consequence of Rosenfeld--Deaett and Tutte--Berge]
\label{prop:matchingmsr}
For every triangle-free graph $F$ of order $n$, including graphs
with isolated vertices,
\[
 \msr(F)=n-\mu(F)=\chi(\overline F).
\]
The same value is obtained if complex Hermitian Gram matrices are allowed.
\end{corollary}
\begin{proof}
A maximum matching and its unmatched vertices partition $F$
into $n-\mu(F)$ cliques. Assigning one orthogonal unit vector to
each part gives the upper bound, since zeros on edges are allowed.
For the reverse inequality, let $M$ be an admissible real or
Hermitian Gram matrix and delete any $S\subseteq V(F)$. Let
$\operatorname{odd}(F-S)$ denote the number of odd-order components
of the remaining graph. Applying the Rosenfeld--Deaett bound to
its components $F_i$ gives
\[
 \rk M\ge\sum_i\left\lceil\frac{|V(F_i)|}{2}\right\rceil
 =\frac{n-|S|+\operatorname{odd}(F-S)}2.
\]
The Tutte--Berge formula~\cite{Berge} maximizes the right side
to $n-\mu(F)$. Finally, every clique of a triangle-free graph
is an edge or a vertex, so $\chi(\overline F)=n-\mu(F)$.
\end{proof}

\begin{theorem}[Hamilton cycles, matchings and semidefinite ranks]
\label{thm:exactmsr}
If $q\ge3$ is a prime power and $t\ge3$, then $G(q,t)$ is
Hamiltonian. For every prime power $q\ge2$ and integer $t\ge2$,
\[
 \mu(G(q,t))=\left\lfloor\frac{q^3t}{2}\right\rfloor,
 \qquad
 \msr(G(q,t))=\left\lceil\frac{q^3t}{2}\right\rceil.
\]
If $q,t$ are odd and $t\ge3$, then $G(q,t)$ is factor-critical:
deleting any vertex leaves a perfect matching. For $t=1$, the
graph is edgeless and $\msr(G(q,1))=q^3$.
\end{theorem}
\begin{proof}
Identify a flag with $(a,x,y)\in\F_q^3$, representing the point
$(x,y)$ and its incident line of slope $a$. With multiplication
\[
 (a,x,y)(b,h,z)=(a+b,x+h,y+z+ah),
\]
these triples form the upper-unitriangular group $U_q$ of order
$q^3$. Put $s(b,h)=(b,h,bh)$ and
$\Sigma_q=\{s(b,h):b,h\ne0\}$.
The right Cayley digraph on $(U_q,\Sigma_q)$ has arcs $g\to gs$
for $s\in\Sigma_q$.
Right multiplication by $s(b,h)$ sends a flag to
$(a+b,x+h,y+(a+b)h)$: its line passes through the old point,
and both its slope and its point change. Thus $D_{B_q}$ is
the right Cayley digraph of $(U_q,\Sigma_q)$.

For $q>2$, the identities
\[
 s(b,h)s(b',h)^{-1}=(b-b',0,0),\qquad
 s(b,h)^{-1}s(b,h')=(0,h'-h,0)
\]
show that $\Sigma_q$ generates both coordinate subgroups,
because every field element is a difference of nonzero elements.
Their commutators give all $(0,0,ax)$, so they generate $U_q$.
By Witte's theorem \cite{Witte}, the directed base has a Hamilton
cycle $e_i\to e_{i+1}$, with indices modulo $m=q^3$.
Lemma~\ref{lem:hamilton-lift} now proves Hamiltonicity for $t\ge3$
and hence gives a matching of size $\lfloor mt/2\rfloor$.
When $mt$ is odd, deleting any vertex from this Hamilton cycle
leaves an even path, proving factor-criticality.
For $t=2$, the edges $(e_i,1)(e_{i+1},2)$ form a perfect matching.
This also completes the Hamiltonicity assertion in
Theorem~\ref{thm:affine}.

For $q=2$, the sole generator $s(1,1)$ has order four, so
$D_{B_2}$ is the disjoint union of two directed four-cycles.
Each component gives a Hamiltonian lift of even order $4t$
when $t\ge3$, and the same matching works for $t=2$.
Thus $G(2,t)$ has a perfect matching for every $t\ge2$.
The semidefinite-rank assertions follow from
Corollary~\ref{prop:matchingmsr} and triangle-freeness.
\end{proof}

A graph is \emph{uniquely colorable} if its partition into
$\chi(G)$ independent color classes is unique up to permuting colors
\cite{HHR}; graph automorphisms do not identify different partitions.
The exact coloring formula from Theorem~\ref{thm:affine}
allows us to characterize uniqueness without further spectral analysis.

\begin{corollary}[Unique optimal colorings]
For every prime power $q\ge2$ and integer $t\ge1$, the graph $G(q,t)$
is uniquely colorable if and only if $t<q$. When $t<q$, every maximum
independent set is a full layer, and every optimal coloring is the layer
coloring up to permutation. For $q\ge4$ and $3\le t<q$, these graphs
are also Hamiltonian and triangle-free.
\end{corollary}
\begin{proof}
The case $t=1$ is immediate. If $2\le t<q$, then $q>2$, and the
generation argument in the proof of Theorem~\ref{thm:exactmsr} shows
that $D_{B_q}$ is strongly connected. Indeed, in a finite group,
inverses of the generators are positive powers of the same generators.
Proposition~\ref{prop:independence-endpoints}, with $\delta=q$, implies
that every maximum independent set selects each flag once, with a
layer index that cannot increase along an arc. Strong connectivity
forces all indices to agree. Each class in an optimal $t$-coloring
has size at most $q^3$, so all $t$ classes have that size and must be
full layers.

For $t\ge q$, both the line-slope coloring and the point first-coordinate
coloring use $q$ colors. An arc changes both coordinates: distinct points
on a nonvertical line have distinct first coordinates. The two colorings
give different partitions, since flags with the same slope may have
different first coordinates. Both are optimal by
Theorem~\ref{thm:affine}, so uniqueness fails. Hamiltonicity in the
stated range follows from Theorem~\ref{thm:exactmsr}.
\end{proof}

\subsection*{The Lov\'asz number}

Write $M^*$ for the conjugate transpose of a complex matrix or vector.
A matrix is Hermitian if $M=M^*$; it is positive semidefinite, denoted
$M\succeq0$, if $x^*Mx\ge0$ for every complex vector $x$. The programs below use real symmetric
variables. Let $\langle J,Z\rangle$ denote the sum of the entries
of $Z$. For a graph $F$ with at least one vertex, the standard theta formulation
\cite{Lovasz} is
\[
 \vartheta(F)=\max\{\langle J,Z\rangle:Z\succeq0,
       \operatorname{tr}Z=1,\ Z_{uv}=0\text{ for }uv\in E(F)\}.
\]
The dual program is
\[
 \vartheta(F)=\min\{\gamma:\gamma I-J+Y\succeq0,
       \ Y=Y^{\mathsf T},\ Y_{uv}=0\text{ for }uv\notin E(F)\}.
\]
The diagonal entries of $Y$ are zero. The primal feasible set is compact, and both programs are strictly
feasible: use $I/|V(F)|$ and $(\gamma,0)$ with $\gamma>|V(F)|$.
Standard semidefinite duality therefore gives the equality of optima
and attainment used below.

We first collect the standard comparisons used in the applications.
A graph is vertex-transitive (respectively, edge-transitive) if its
automorphisms act transitively on its vertices (respectively, edges).
For every graph $F$ with at least one vertex,
$\vartheta(F)\vartheta(\overline F)\ge|V(F)|$, with equality
when $F$ is vertex-transitive, and $\vartheta(\overline F)\le\chi(F)$;
see Knuth~\cite[Sections~4, 23 and 25]{Knuth}.

Let $D$ be a loopless digraph on $m\ge1$ vertices, put
$F=F(D)$, $\eta=\vartheta(F)$ and $H=H_t(D)$, and take $t\ge1$.
The independent $t$-fold blow-up replaces each vertex by $t$
nonadjacent copies and each edge by all edges between its two classes;
the clique blow-up instead makes each class a clique.
The graph $H$ is a spanning subgraph of the first blow-up of $F$,
and its complement contains the second blow-up of $\overline F$. The two standard blow-up identities
\cite[Section~16, (16.8)--(16.9)]{Knuth}, monotonicity and layer coloring give
\begin{equation}\label{eq:theta-transfer}
 \vartheta(H)\ge\max\{m,t\eta\},\qquad
 \frac{mt}{\vartheta(H)}\le\vartheta(\overline H)
 \le\min\{t,\vartheta(\overline F)\}.
\end{equation}
With $s=\rk_{\mathbb R}[A_D-I\ \ A_D^{\mathsf T}-I]$,
Theorem~\ref{thm:rank-comparison} also gives
$\vartheta(H)/\rho(H)\ge\max\{m,t\eta\}/[m+(t-1)s]$.
In particular, if $\eta>s$, then $\vartheta(H)>\rho(H)$
for $t>(m-s)/(\eta-s)$. The triangle-free condition follows separately from Lemma~\ref{lem:lift}.

Arc transitivity means that an automorphism can send any arc to any
other. Suppose $D$ has this property, and put $A=A_D$. Symmetry averaging~\cite[Sections~2--3]{dKPS}
allows us to restrict the dual exactly to matrices
\begin{equation}\label{eq:averaged-dual}
 \gamma I_{mt}-J_t\otimes J_m+
 U\otimes A+U^{\mathsf T}\otimes A^{\mathsf T}\succeq0,
\end{equation}
where $U$ is any real strictly upper triangular $t\times t$ matrix.
Indeed, an averaged upper layer block has one value on all arcs and
vanishes elsewhere; conversely, every such $U$ has the required edge
support. Its entries need not be nonnegative.

Write $\lambda_{\min}(S)$ for the least eigenvalue of a real symmetric
matrix $S$. We first use this reduction to decide when theta equals
the size of one layer.

\begin{proposition}[Theta equality for ordered lifts]
\label{prop:general-arc-theta}
Let $D$ be a loopless digraph on $m$ vertices in which every vertex
has indegree and outdegree $d>0$.  Suppose an automorphism group of
$D$ is transitive on its arcs. Set $A=A_D$, $F=F(D)$,
$S=(A+A^{\mathsf T})/2$ and $\beta=-\lambda_{\min}(S)>0$.
Then
\[
 \vartheta(F)=\frac{m\beta}{d+\beta},\qquad
 \vartheta(H_t(D))\ge
 \max\left\{m,\frac{mt\beta}{d+\beta}\right\}
 \quad(t\ge1).
\]
Moreover, for every $t\ge2$,
\[
 \vartheta(H_t(D))=m
 \quad\Longleftrightarrow\quad
 \lambda_{\min}(A_{H_t(D)})=-d.
\]
Whenever these equivalent conditions hold,
$\alpha(H_t(D))=\vartheta(H_t(D))=m$.
Neither connectedness nor the exclusion of transitive triangles
is required for these statements.
\end{proposition}
\begin{proof}
Arc transitivity implies that either every arc has its reverse or
none does. Thus $F$ has adjacency matrix $S$ or $2S$ and is regular
and edge-transitive. Lov\'asz's formula~\cite[Theorem~9]{Lovasz}
gives the first equality, and \eqref{eq:theta-transfer}
gives the lift lower bound.

For the criterion, put $B=A_{H_t(D)}$. At $\gamma=m$, the compression
of \eqref{eq:averaged-dual} to normalized layer constants has zero
diagonal and off-diagonal entries $du_{ij}-m$. Positive
semidefiniteness forces every entry to vanish, so $u_{ij}=m/d$.
The only possible averaged dual is therefore $Y=(m/d)B$.

Conversely, this choice is feasible exactly when
$M=mI_{mt}-J_{mt}+(m/d)B\succeq0$; the independent layer then
forces equality. Since $B$ is $d(t-1)$-regular, $M$ vanishes on
$\mathbf1_{mt}$ and equals $m(I+B/d)$ on its orthogonal
complement. The layer-constant vectors whose coefficients sum to
zero have eigenvalue $-d$ for $B$. Therefore $M\succeq0$ exactly
when $\lambda_{\min}(B)=-d$, proving the criterion and
$\alpha(H_t(D))=m$ in the equality range.
\end{proof}

\begin{lemma}[Affine symmetry]\label{lem:affine-arc-transitive}
For every prime power $q\ge2$, the affine transformations
\[
 x'=rx+b,\qquad y'=sy+ux+v,
 \qquad r,s\in\F_q\setminus\{0\},\quad b,u,v\in\F_q,
\]
act transitively on the vertices and on the arcs of $D_{B_q}$.
Their simultaneous action in all layers preserves $G(q,t)$.
\end{lemma}
\begin{proof}
These transformations preserve incidence and nonvertical lines,
and therefore preserve the directed edge relation.
Consider an arc from the flag with point $(x_0,y_0)$ and line
of slope $a$ to the flag with point $(x_1,y_1)$ and line of
slope $a'$. The second line passes through the first point.
The arc conditions imply $\Delta x=x_1-x_0\ne0$ and
$\Delta a=a'-a\ne0$. The transformation
\[
 X=\frac{x-x_0}{\Delta x},\qquad
 Y=\frac{y-y_0-a(x-x_0)}{\Delta a\,\Delta x}
\]
maps these flags to $((0,0),Y=0)$ and $((1,1),Y=X)$.
Thus all arcs lie in one orbit. Every vertex has positive outdegree,
so the action is also transitive on vertices. The assertion for
the lift follows from its definition.
\end{proof}

We use an orthonormal form of the two-dimensional spaces from Section~3.

\begin{lemma}[The affine spectral blocks]\label{lem:affine-blocks}
Let $q\ge3$ be a prime power, let $A=A_{D_{B_q}}$, and put
\[
 d_0=(q-1)^2,\qquad
 a_q=\frac{q-2}{2},\qquad
 h_q=\frac{(q-2)\sqrt q}{2},\qquad
 \tau_q=\frac{q\sqrt{q-1}}2.
\]
There is a common orthogonal decomposition for $A$ and $A^{\mathsf T}$
into scalar spaces with values $d_0$, $1$ and $1-q$, and
two-dimensional spaces on which, in a suitable orthonormal basis,
\begin{equation}\label{eq:affine-normal-form}
 A=\begin{pmatrix}-a_q&h_q+\tau_q\\h_q-\tau_q&-a_q\end{pmatrix}.
\end{equation}
The scalar multiplicities are $1$, $q^3-2q^2+1$ and $2(q-1)$,
respectively, and there are $q(q-1)$ two-dimensional spaces.
The value $d_0$ occurs on the all-ones direction; all other spaces
are orthogonal to it. In particular, if
$S=(A+A^{\mathsf T})/2$ and
\[
 \beta_q=\max\left\{q-1,\frac{(q-2)(1+\sqrt q)}2\right\},
\]
then $\lambda_{\min}(S)=-\beta_q$.
For $q\ge4$, the second term in this maximum equals $\beta_q$.
\end{lemma}
\begin{proof}
The proof of Theorem~\ref{lem:singular-blocks} gives the stated scalar actions.
On a pair with $\sigma=\sqrt q$, the symmetric action is
\[
 I+\frac12(Q_\sigma+Q_\sigma^{\mathsf T})M_\sigma
   =\begin{pmatrix}-a_q&h_q\\h_q&-a_q\end{pmatrix},
\]
where these coordinates have Gram matrix $M_\sigma$.
The displayed symmetric matrix commutes with $M_\sigma$, so
orthonormalizing by $M_\sigma^{-1/2}$ leaves it unchanged.
The skew-adjoint action has determinant
$\sigma^2\det M_\sigma/4=q^2(q-1)/4=\tau_q^2$.
In an orthonormal basis the latter is
$\left(\begin{smallmatrix}0&\tau_q\\-\tau_q&0\end{smallmatrix}\right)$
after choosing orientation. A rotation puts the symmetric part
in the displayed form, proving \eqref{eq:affine-normal-form}.
Its symmetric eigenvalues are $-a_q\pm h_q$; comparison with
the scalar values gives $\beta_q$ and the assertion for $q\ge4$.
\end{proof}

We next construct a feasible theta dual certificate.
Let $D$ be loopless with $m$ vertices and constant indegree and outdegree $d>0$.
Put $A=A_D$ and $A_0=A|_{\mathbf1_m^\perp}$; this subspace is
invariant under both $A$ and $A^{\mathsf T}$.
For symmetric $K=(k_{ij})\in\mathbb R^{t\times t}$ with unit diagonal,
let $M_t(K)$ have diagonal blocks $dI$, upper blocks $k_{ij}A_0$,
and transpose blocks below. For $\ell>0$,
\begin{equation}\label{eq:weighted-dual}
 \ell K-J_t\succeq0,\quad M_t(K)\succeq0
 \quad\Longrightarrow\quad
 \vartheta(H_t(D))\le m\ell.
\end{equation}
Indeed, choose $\gamma=m\ell$, upper blocks
$Y_{ij}=(m\ell/d)k_{ij}A$, transpose blocks below, and zero diagonal
blocks. The dual restricts to $m(\ell K-J_t)$ on layer constants
and to $(m\ell/d)M_t(K)$ on their orthogonal complement.

We use the following scalar kernel calculation. For $0<r<1$,
the Kac--Murdock--Szeg\H{o} matrix $K_r=(r^{|i-j|})_{i,j=1}^t$
is positive definite. Its classical tridiagonal inverse gives
\begin{equation}\label{eq:kms-certificate}
 \ell_r=\mathbf1_t^{\mathsf T}K_r^{-1}\mathbf1_t
       =\frac{t(1-r)+2r}{1+r},\qquad \ell_rK_r-J_t\succeq0.
\end{equation}
Here the inverse formula is from Higham~\cite[Section~8.2]{Higham};
the last inequality is Cauchy--Schwarz, valid for any positive
definite $K$ with $\ell=\mathbf1^{\mathsf T}K^{-1}\mathbf1$.

In the Fourier calculation below, $\mathrm{i}$ denotes the imaginary unit.

\begin{theorem}[Sharp theta bounds for affine lifts]
\label{thm:theta-bounds}
Let $q\ge3$ be a prime power, put $m=q^3$ and $d=(q-1)^2$,
and let $F_q=F(D_{B_q})$.
With $\beta_q$ as in Lemma~\ref{lem:affine-blocks}, put
$b_q=\beta_q/(d+\beta_q)$.
Then $\vartheta(F_q)=mb_q$ and, for every integer $t\ge1$,
\[
 \max\{m,mtb_q\}\le\vartheta(G(q,t))
 \le m\bigl(1+(t-1)b_q\bigr).
\]
In particular,
\begin{equation}\label{eq:theta-exact-center}
 0\le\vartheta(G(q,t))-t\vartheta(F_q)
 \le q^3-\vartheta(F_q).
\end{equation}
For each fixed $q$, this additive constant is best possible in a
bound valid for every $t\ge1$, since equality holds at $t=1$.
The weaker bound with right side $q^3$ has best possible uniform
coefficient $1$ as $q$ and $t$ vary.
Moreover, with an absolute implied constant,
\begin{equation}\label{eq:theta-quantitative}
 \vartheta(G(q,t))
 =\frac12tq^{5/2}+\frac14tq^2
   +O\!\left(tq^{3/2}+q^3\right).
\end{equation}
Consequently, if $q\to\infty$ and $t/\sqrt q\to\infty$,
then $\vartheta(G(q,t))=(1/2+o(1))tq^{5/2}$.
\end{theorem}
\begin{proof}
The symmetry and spectral lemmas, together with
Proposition~\ref{prop:general-arc-theta}, give $\vartheta(F_q)$
and the lower bound. The case $t=1$ is immediate, so assume $t\ge2$.
Put $A_0=A_{D_{B_q}}|_{\mathbf1_m^\perp}$ and use the certificate
\eqref{eq:weighted-dual}. It remains to find a unit-diagonal $K$
with $\ell K-J_t\succeq0$ and $M_t(K)\succeq0$.
For a Toeplitz matrix $K=(k_{|i-j|})$ with $k_0=1$ and
absolutely summable coefficients, Fourier integration proves the
latter condition if
\[
 dI+\frac{zA_0+\overline zA_0^{\mathsf T}}2\succeq0,
 \qquad z=2\sum_{j\ge1}k_je^{\mathrm{i}j\theta}
 \quad(\theta\in\mathbb R).
\]
Indeed, call this symbol $\mathsf H(\theta)$ and put
$p(\theta)=\sum_{j=1}^t e^{-\mathrm{i}j\theta}v_j$ for a block vector $v$.
Fourier orthogonality gives
$v^*M_t(K)v=(2\pi)^{-1}\int_0^{2\pi}
 p(\theta)^*\mathsf H(\theta)p(\theta)\,d\theta$.
Thus the pointwise condition implies finite-matrix positivity
for every $t$, not merely asymptotically.
Only the scalar values $1,1-q$ and the mixed block remain.
Writing $z=x+\mathrm{i}y$, abbreviate
$a=a_q$, $h=h_q$, $\tau=\tau_q$, and $\beta=\beta_q$.
The mixed condition is
\begin{equation}\label{eq:theta-symbol-condition}
 d-ax\ge\sqrt{h^2x^2+\tau^2y^2}.
\end{equation}
Both scalar conditions hold whenever $-1<x\le d/\beta\le q-1$.

For $q=3$ and $q\ge5$, a single geometric kernel suffices.
Take $K=K_r=(r^{|i-j|})$
with $r=d/(d+2\beta)$ and $\ell=1+(t-1)b_q$;
\eqref{eq:kms-certificate} verifies the layer-constant condition.
Put $c=d/\beta$ and $s=d/(d+\beta)$.
The symbol $z=2re^{\mathrm{i}\theta}/(1-re^{\mathrm{i}\theta})$
traces the circle $y^2=(c-x)(x+s)$, with $-s\le x\le c$.
For $q=3$, the difference of the squares in
\eqref{eq:theta-symbol-condition} is $4x^2-10x+10>0$.
For $q\ge5$, we have $\beta=a+h$ and
$\tau^2-h^2+a^2=d$. The same difference is
$(c-x)(d\beta-\tau^2s-dx)$, which is nonnegative because
\[
 \begin{aligned}
 (d+\beta)(\beta^2-d)-\beta\tau^2
 =\frac{q(1+\sqrt q)^2}{4}
 \bigl((q-5)^2(q+1)+4(q-1)\sqrt q-33\bigr)>0.
 \end{aligned}
\]
The final inequality follows from $16\sqrt5>33$.
Also $d-ax>0$ throughout the circle.
This verifies every nonconstant block and proves the upper bound
for $q\ne4$.

The case $q=4$ requires a different certificate: the analogous
single geometric kernel fails the mixed-symbol inequality near
$\theta=0$. For $t=1,2$, theta equals $64$, by an independent
layer and, for $t=2$, the matching-edge clique cover from
Theorem~\ref{thm:exactmsr}. Assume $t\ge3$ and set
$K=\tfrac{21}{22}K_{4/7}+\tfrac1{22}K_{5/6}$.
We will increase its endpoint correlations while retaining a
feasible dual. Its symbol is
$z=\zeta(49-40\zeta)/[(7-4\zeta)(6-5\zeta)]$, where
$\zeta=e^{\mathrm{i}\theta}$, and $-89/121\le x\le3$.
In a basis diagonalizing the symmetric part, the mixed symbol is
\[
 H(\theta)=
 \begin{pmatrix}9+x&\mathrm{i}\sqrt{12}\,y\\
                 -\mathrm{i}\sqrt{12}\,y&9-3x\end{pmatrix}.
\]
The scalar symbols are its diagonal entries.
Put $u=1-\cos\theta$, $v=2-u$, $D=(9+56u)(1+60u)$,
and $g(\theta)=40u/(8+u)$.
Define the polynomial
\[
 \begin{aligned}
 P(u)={}&1728+8188728u+224797186u^2\\
        &+u^3(169589409+1994758080v+2562739200v^2).
 \end{aligned}
\]
The first diagonal entry of $H-gI_2$ is at least $32/121>0$.
Direct expansion gives $\det(H-gI_2)=4uP(u)/[D^2(8+u)^2]\ge0$,
since $0\le u\le2$ and $v=2-u\ge0$. Thus every nonconstant
symbol is at least $gI$.

Let $e_1,e_t$ be the first and last coordinate vectors in $\mathbb R^t$.
Set $E=e_1e_1^{\mathsf T}+e_te_t^{\mathsf T}$, $R=I_t+E/7$,
$r_0=9-4\sqrt5$, and $h=20-8\sqrt5>135/64$.
The Toeplitz matrix of the Fourier coefficients of $g$ is
$T_t(g)=40I_t-16\sqrt5K_{r_0}$. It has negative
off-diagonal entries and row sums $h(r_0^{i-1}+r_0^{t-i})$.
Subtracting their diagonal leaves a weighted Laplacian, so
$T_t(g)\succeq hE$ and $M_t(K)\succeq hE\otimes I$.
For $\widehat K=RKR+I_t-R^2$, the same matrix becomes
\[
 \begin{aligned}
 M_t(\widehat K)
 &=(R\otimes I)M_t(K)(R\otimes I)+9(I_t-R^2)\otimes I\\
 &\succeq \frac{64}{49}
       \left(h-\frac{135}{64}\right)E\otimes I\succeq0.
 \end{aligned}
\]

For the layer constants, put $\widehat K_r=RK_rR+I_t-R^2$
for $r\in\{4/7,5/6\}$. These endpoint-modified kernels require
the following short calculation, rather than the unmodified inverse. Let $a_1=a_{t-1}=8r/7$ and $a_i=r$
otherwise. The lower bidiagonal matrix $L$ with diagonal entries
one and subdiagonal entries $-a_i$ satisfies
$L\widehat K_rL^{\mathsf T}=\operatorname{diag}(1,1-a_1^2,\ldots,1-a_{t-1}^2)$.
Since $0<a_i<1$, this proves positive definiteness and gives
\[
 \ell_r:=\mathbf1_t^{\mathsf T}\widehat K_r^{-1}\mathbf1_t
 =1+(t-3)\frac{1-r}{1+r}
       +2\frac{1-8r/7}{1+8r/7}.
\]
In particular, $\ell_{4/7}=(243t+536)/891$ and
$\ell_{5/6}=(41t+350)/451$. Cauchy--Schwarz for these two
positive definite matrices now gives
\[
 \widehat K\succeq
 \left(\frac{21}{22\ell_{4/7}}+\frac1{22\ell_{5/6}}\right)J_t
 =\frac{39852t+308663}{(243t+536)(41t+350)}J_t
 \succeq\frac4{t+3}J_t.
\]
The last comparison reduces to $115t+175589>0$.
The weighted dual with $\ell=(t+3)/4$ therefore gives
$\vartheta(G(4,t))\le16t+48$, completing
\eqref{eq:theta-exact-center} for every $q\ge3$.
Equality at $t=1$ proves fixed-$q$ sharpness, while
$1-b_q\to1$ proves the uniform coefficient claim.

For $q\ge4$, direct simplification gives
\[
 b_q=\frac1{2\sqrt q}+\frac1{4q}
       -\frac{3\sqrt q-2}{4q(2q-\sqrt q-2)}.
\]
The last term has magnitude at most $q^{-3/2}/2$.
Together with \eqref{eq:theta-exact-center}, this gives
\eqref{eq:theta-quantitative}; the case $q=3$ is absorbed using
$\vartheta(F_3)=9$. The leading asymptotic follows.
\end{proof}

\Needspace{12\baselineskip}

\begin{theorem}[The exact layer cutoff]\label{thm:theta-cutoff}
For every prime power $q\ge3$, use the principal values of the
inverse trigonometric functions and put
\[
 \phi_q=\arcsin\left(\frac{\sqrt q}{2(q-1)}\right),
 \qquad
 \omega_q=\arctan\left(
       \frac{q}{q-2}\tan\left(\frac\pi4-\frac{\phi_q}{2}\right)
                        \right),
\]
Put $T_q=\omega_q/\phi_q$. Then $2<T_q\le q$, and for every positive integer $t$,
\[
 \vartheta(G(q,t))=q^3
 \quad\Longleftrightarrow\quad
 t\le T_q.
\]
In this range, $\alpha(G(q,t))=q^3$ as well.
Moreover,
\begin{equation}\label{eq:theta-cutoff-expansion}
 T_q=\frac\pi2\sqrt q-\frac12
       +\left(2-\frac{25\pi}{48}\right)q^{-1/2}
       +O(q^{-3/2}).
\end{equation}
Thus the largest admissible integer number of layers is exactly
$\lfloor T_q\rfloor$.
\end{theorem}
\begin{proof}
Write $c=(q-1)^2$, $a=a_q$, $h=h_q$, and $\tau=\tau_q$.
On a mixed space of Lemma~\ref{lem:affine-blocks}, write
$A=\left(\begin{smallmatrix}-a&h+\tau\\h-\tau&-a\end{smallmatrix}\right)$.
Put $S=(A+A^{\mathsf T})/2$, $R=cI_2-S$, and
$E=\left(\begin{smallmatrix}0&1\\-1&0\end{smallmatrix}\right)$.
The matrix $R$ is positive definite: its trace is positive and
$\Delta:=\det R=q(q-1)(4q^2-9q+4)/4>0$.
Let $K_t$ be the restriction of $cI+A_{G(q,t)}$ to this mixed
space over all layers. Congruence by the block lower bidiagonal
matrix with diagonal $I_2$ and subdiagonal $-I_2$ gives a block
tridiagonal matrix with diagonal
$cI_2,2R,\ldots,2R$ and upper off-diagonal blocks $-R+\tau E$.

Congruence by $R^{-1/2}$ in each layer gives diagonal
$cR^{-1},2I_2,\ldots,2I_2$ and upper blocks $-I_2+rE$, where
\[
 r=\frac{\tau}{\sqrt\Delta}
   =\sqrt{\frac{q}{4q^2-9q+4}},\qquad s=\sqrt{1+r^2}.
\]
The matrix $V=(I_2-rE)/s$ is orthogonal. Using $V^{-(i-1)}$
in layer $i$ changes these upper blocks to $-sI_2$ and leaves
the diagonal blocks unchanged. Diagonalizing the first block
then splits $K_t$ into two scalar tridiagonal matrices, with
off-diagonal entries $-s$, later diagonal entries $2$, and first
entries $c/(c+a+h)$ and $c/(c+a-h)$. The latter is larger,
so only the first matrix needs to be tested.
Put $b=c/(c+a+h)\in(0,1)$.

The leading determinants satisfy $D_0=1$, $D_1=b$ and
$D_j=2D_{j-1}-s^2D_{j-2}$. With $\phi=\arctan r$ and
$\omega=\arctan(r/(1-b))$, their solution is
\[
 D_j=s^j\frac{\sin(\omega-j\phi)}{\sin\omega}.
\]
Since $0<\phi<\omega<\pi/2$, all leading determinants are
positive for $j<\omega/\phi$. If the ratio is an integer,
the pivots $D_j/D_{j-1}$ are positive until the last, which is
zero, so the boundary matrix is positive semidefinite.
At the first integer beyond the ratio the sine is negative.
Thus every larger matrix contains a principal submatrix with
negative determinant and is not positive semidefinite.
Thus $K_t\succeq0$ if and only if $t\le\omega/\phi$.
The identities
\[
 \sin\phi=\frac{\sqrt q}{2(q-1)},\qquad
 \frac r{1-b}=\frac q{q-2}
     \sqrt{\frac{2q-\sqrt q-2}{2q+\sqrt q-2}}
 =\frac q{q-2}\tan\left(\frac\pi4-\frac\phi2\right)
\]
identify this ratio as $T_q$. Moreover, $D_2>0$ is equivalent
to $(q-2)(2q-\sqrt q-2)>0$, so $T_q>2$.

The scalar base values $1$, $1-q$, and $c$ give the blocks
$(c-1)I_t+J_t$, $(q-1)(qI_t-J_t)$, and $cJ_t$.
Their only additional condition is $t\le q$, which is automatic
in the mixed equality range. Indeed, for $q\ge4$, testing $K_q$
on a layer-constant vector in the $-(a+h)$ eigenspace of $S$
gives $c-(q-1)(a+h)\le0$, so $K_q$ is not positive definite. Since the mixed matrices
are positive definite for every integer $t<T_q$, we obtain
$T_q\le q$, including the zero Rayleigh value when $q=4$.
For $q=3$,
$D_3=-32\sqrt3/(13(9+\sqrt3))<0$, so $T_3<3$.
Proposition~\ref{prop:general-arc-theta} now proves the cutoff
for $t\ge2$; $t=1$ is immediate. In the equality range,
$q^3\le\alpha(G(q,t))\le\vartheta(G(q,t))=q^3$.

Finally, the tangent addition formula gives
\[
 \omega_q=\frac\pi4-\frac{\phi_q}{2}
 +\arctan\left(\frac{\cos\phi_q}{q-1-\sin\phi_q}\right).
\]
For $z=q^{-1/2}$, the two small terms satisfy
\[
 \phi_q=\frac z2+\frac{25z^3}{48}+O(z^5),\qquad
 \arctan\left(\frac{\cos\phi_q}{q-1-\sin\phi_q}\right)
       =z^2+O(z^4).
\]
Division by $\phi_q$ gives \eqref{eq:theta-cutoff-expansion}.
\end{proof}

The finite formula includes the critical scale: equality is decided
by $t\le\lfloor T_q\rfloor$, even when $t/\sqrt q\to\pi/2$.
For example, $\lfloor T_3\rfloor=2$, $\lfloor T_5\rfloor=3$,
and $\lfloor T_{16}\rfloor=5$.
The independence transition is instead exactly $t=q$, by
Theorem~\ref{thm:affine}.

The equality range also determines coloring and zero-error capacity.
For a nonempty graph $H$ and integer $r\ge1$, its strong power
$H^{\boxtimes r}$ has vertex set $V(H)^r$; distinct tuples are
adjacent when their entries are equal or adjacent in every coordinate.
The Shannon capacity~\cite{Shannon} is
$\Theta(H)=\sup_{r\ge1}\alpha(H^{\boxtimes r})^{1/r}$.
We use the classical bounds
$\alpha(H)\le\Theta(H)\le\vartheta(H)$ of Lov\'asz~\cite{Lovasz}
and $\Theta(H)\le\rho_{\mathbb K}(H)$ from
Haemers' rank bound~\cite{Haemers}.
Together with \eqref{eq:theta-transfer}, these bounds translate
the preceding theta estimates into classical graph parameters.

\begin{corollary}[Coloring, capacity and the complementary theta value]
\label{cor:classical-parameters}
Let $q\ge3$ be a prime power, $t\ge1$, $H=G(q,t)$,
$m=q^3$ and $\eta_q=\vartheta(F_q)$. The exact independence and chromatic numbers are those in
Theorem~\ref{thm:affine}. Moreover,
\[
 \frac{mt}{m+(t-1)\eta_q}
 \le\vartheta(\overline H)
 \le\min\left\{t,\frac m{\eta_q}\right\},\qquad
 1\le\frac{\vartheta(H)\vartheta(\overline H)}{mt}
 \le1+\frac{m-\eta_q}{t\eta_q}.
\]
If $t\le T_q$, then
\[
 \alpha(H)=\Theta(H)=\vartheta(H)=q^3,
 \qquad \vartheta(\overline H)=\chi(H)=t.
\]
In this range, $\alpha(H^{\boxtimes r})=q^{3r}$ for every $r\ge1$.
For $q\ge5$, $\vartheta(H)=\alpha(H)$ holds if and only if
$t\le T_q$.
For fixed $q$,
$\vartheta(\overline{G(q,t)})\to q^3/\eta_q$ as $t\to\infty$. If $q\to\infty$ and $t/\sqrt q\to\infty$, then
$\vartheta(\overline{G(q,t)})=(2+o(1))\sqrt q$.
\end{corollary}
\begin{proof}
By Lemma~\ref{lem:affine-arc-transitive}, the base $F_q$ is
vertex-transitive, so $\vartheta(\overline{F_q})=m/\eta_q$.
Substituting Theorem~\ref{thm:theta-bounds} into
\eqref{eq:theta-transfer} gives the complementary theta bounds.
The product upper bound follows from
$\vartheta(H)\le m+(t-1)\eta_q$ and
$\vartheta(\overline H)\le m/\eta_q$; the product lower bound is classical.
When $t\le T_q$, Theorem~\ref{thm:theta-cutoff} gives
$\alpha(H)=\vartheta(H)=m$. The classical capacity bounds give
$\Theta(H)=m$, and $t=|V(H)|/\vartheta(H)\le
\vartheta(\overline H)\le\chi(H)\le t$ gives the coloring equalities.
Products of a maximum independent set provide $m^r$ vertices in
every strong power; the definition of $\Theta$ gives the reverse bound.
For $q\ge5$, we have $\beta_q>q-1$ and hence $\eta_q>q^2$.
If $T_q<t\le q$, the cutoff theorem gives $\vartheta(H)>m=\alpha(H)$;
if $t>q$, then $\vartheta(H)\ge t\eta_q>tq^2=\alpha(H)$.
This proves the additional equivalence.
The fixed-$q$ limit follows by squeezing, while
$\eta_q=(1/2+o(1))q^{5/2}$ and $t/\sqrt q\to\infty$
give the last asymptotic formula.
\end{proof}

For $q=2$ and $t\ge2$, Theorem~\ref{thm:affine} gives
$\alpha(G(2,t))=4t$ and $\chi(G(2,t))=2$.
The matching in Theorem~\ref{thm:exactmsr} supplies a cover by
$4t$ cliques, so the Lov\'asz sandwich inequalities~\cite{Lovasz}
also give $\Theta(G(2,t))=\vartheta(G(2,t))=4t$.
For $t=1$, both values are $8$.

\subsection*{Prescribed orders and comparison of the parameters}

If $H$ is an induced subgraph of $G$, then $A_H+I$ is a principal
submatrix of $A_G+I$, so $\rho_{\mathbb K}(H)\le\rho_{\mathbb K}(G)$
over every field. The Lov\'asz number and independence number are
also monotone under induced subgraphs, while chromatic number is
nonincreasing on passage to an induced subgraph.
We retain complete layers and insert pairs of vertices into a spanning
cycle, preserving the parameter asymptotics and the prescribed order.

\begin{proposition}[A parameterized construction at every order]
\label{prop:all-orders-parameter}
For every fixed $a>0$ and integer $n\ge1$, there is a triangle-free
graph $H_{n,a}$ of order $n$ with
$\msr(H_{n,a})=\chi(\overline{H_{n,a}})=\lceil n/2\rceil$.
For every sufficiently large $n$, the graph $H_{n,a}$ is Hamiltonian.
As $n\to\infty$,
\[
 \begin{aligned}
 \rho_{\mathbb K}(H_{n,a})
 &=\bigl((1+2a)a^{-3/4}+o(1)\bigr)n^{3/4},\\
 \vartheta(H_{n,a})
 &=\bigl(\tfrac12a^{1/8}+o(1)\bigr)n^{7/8},\\
 \vartheta(\overline{H_{n,a}})
 &=\bigl(2a^{-1/8}+o(1)\bigr)n^{1/8}.
 \end{aligned}
\]
The choice of vertices is independent of $\mathbb K$, and the rank
error is uniform in that field. In addition,
\[
 \alpha(H_{n,a})=\bigl(\max\{1,a\}a^{-3/4}+o(1)\bigr)n^{3/4},
 \qquad
 \chi(H_{n,a})=\bigl(\min\{1,a\}a^{-1/4}+o(1)\bigr)n^{1/4}.
\]
In particular, $\alpha(H_{n,a})\chi(H_{n,a})=(1+o(1))n$, and
the Shannon capacity $\Theta(H_{n,a})$ has order $n^{3/4}$.
\end{proposition}
\begin{proof}
Let $q$ be the least prime at least $\max\{7,(n/a)^{1/4}\}$.
By the prime number theorem, $q\sim(n/a)^{1/4}$.
We first specify a directed Hamilton cycle in $D_{B_q}$, using
its flag coordinates $(b,x,y)$, where $b$ is the slope, as in
Theorem~\ref{thm:exactmsr}.
In $\mathbb F_q^2$, list $(b,x)=(i,j+i)$ with $j$ and then $i$
running from $0$ to $q-1$, and swap the entries at positions
$q-1$ and $q$, numbered from zero. Join this list cyclically.
Before the swap the steps are $(1,1)$ and $(1,2)$; the three
changed steps are $(2,3),(-1,-2),(2,3)$, so both coordinates
still change. Lift a step to $(b',x')$ by setting
$y'=y+b'(x'-x)$, which is exactly the directed flag relation.
The sum of these increments is zero before the swap and increases
by $5$ after it. Repeating the projected tour $q$ times therefore
visits every flag exactly once and returns to its start, since
$5\ne0$ in the prime field $\mathbb F_q$.

Put $m=q^3$, and let $k$ be the greatest integer at most $n/m$
with $k\equiv n\pmod2$. For large $n$, we have $k\sim aq$ and
$k\ge3$. Set $t=k+2$ and $h=(n-km)/2$, so $0\le h<m$.
Lemma~\ref{lem:hamilton-lift}, applied to this explicit base cycle,
gives a Hamiltonian induced subgraph $H_{n,a}$ of $G(q,t)$ with
$n$ vertices and $k$ complete layers. Its spanning cycle gives a
matching of size $\lfloor n/2\rfloor$, so
Corollary~\ref{prop:matchingmsr} proves the two claimed equalities.

We have induced containments
$G(q,k)\subseteq H_{n,a}\subseteq G(q,t)$, with $k=t-2$.
Set $s_{\mathbb K}=2q^2-1-\varepsilon_{\mathbb K}(q)$.
By Corollary~\ref{cor:affine-common-field} and induced-subgraph
monotonicity, we obtain
\[
 \begin{aligned}
 m+(k-2)s_{\mathbb K}
 &\le\rho_{\mathbb K}(H_{n,a})\le m+(t-1)s_{\mathbb K},\\
 \vartheta(G(q,k))
 &\le\vartheta(H_{n,a})\le\vartheta(G(q,t)).
 \end{aligned}
\]
Since $s_{\mathbb K}=2q^2+O(1)$ uniformly in $\mathbb K$, the
rank is $(1+2a+o(1))q^3$ with a uniform error. By
Theorem~\ref{thm:theta-bounds}, the theta value is
$(a/2+o(1))q^{7/2}$. Substituting $n\sim aq^4$ proves both
asymptotic formulas for rank and theta.
For the complement, induced-subgraph monotonicity and
Corollary~\ref{cor:classical-parameters} give
\[
 \frac{n}{\vartheta(H_{n,a})}
 \le\vartheta(\overline{H_{n,a}})
 \le\vartheta(\overline{G(q,t)})
 \le\frac{q^3}{\vartheta(F_q)}=(2+o(1))\sqrt q.
\]
The first and last expressions have the same leading term, proving
the complementary theta formula. The same induced containments,
now applied to the exact formulas in Theorem~\ref{thm:affine}, give
\[
 \begin{aligned}
 \max\{q^3,kq^2\}&\le\alpha(H_{n,a})\le\max\{q^3,tq^2\},\\
 \min\{k,q\}&\le\chi(H_{n,a})\le\min\{t,q\}.
 \end{aligned}
\]
Since $k=t-2\sim aq$ and $n\sim aq^4$, both pairs have matching
leading terms. This proves the independence and coloring formulas
and their product. The inequalities $\alpha\le\Theta\le\rho$ give the
stated order for capacity.
When the computed value of $k$ is less than three, use a matching and
at most one isolated vertex instead. This affects only finitely many orders.
\end{proof}

\Needspace{5\baselineskip}
\begin{proof}[Proof of Theorem~\ref{thm:main}]
Proposition~\ref{prop:all-orders-parameter} with $a=3/2$ gives
both asymptotic assertions, with rank coefficient $C_*$ and
theta coefficient $\tfrac12(3/2)^{1/8}$. Since $C_*<3$, it remains
to give graphs at every order with rank at most $3n^{3/4}$
and a matching of size $\lfloor n/2\rfloor$.

For $n\ge2401$, let $q$ be the largest prime power not exceeding
$y=n^{1/4}$. We first check that $13y/16\le q\le y$. Indeed, the consecutive
prime powers from $7$ to $25$ are
$7,8,9,11,13,16,17,19,23,25$, with successive ratios at most
$16/13$. If $q\ge25$, Nagura's theorem~\cite[p.~180]{Nagura}
gives a prime $q<p<6q/5$. Maximality of $q$ forces $y<p$,
so $q>5y/6$.
Put $t=\lceil n/q^3\rceil\ge2$. By Theorem~\ref{thm:exactmsr},
we may retain $\lfloor n/2\rfloor$ matching pairs in $G(q,t)$
and one further vertex when $n$ is odd. The induced graph $H$ satisfies
\[
 \rho_{\mathbb K}(H)
 \le q^3+(2q^2-1)(t-1)<q^3+\frac{2n}{q}.
\]
The convex function $f(u)=u^3+2/u$ has
$f(13/16)=3-111/53248$ and $f(1)=3$. Taking
$u=q/n^{1/4}\in[13/16,1]$ proves
$\rho_{\mathbb K}(H)<3n^{3/4}$ over every field.

For $n\le1250$, use a matching and at most one isolated vertex.
For $1251\le n\le1800$, use as a host the $1374$ endpoints of
a largest matching in $G(5,11)$; for $1801\le n\le2400$, use
$G(7,6)$. These hosts have perfect matchings and respective
order--rank bounds $(N,R)=(1374,615)$ and $(2058,828)$.
Retaining matching pairs and, for odd $n$, one further vertex,
or appending a disjoint matching and at most one isolated vertex,
gives the requested order. Writing $x_+=\max\{x,0\}$, the rank
is at most $R+\lceil(n-N)_+/2\rceil$.
For these hosts, the bound $R+\tfrac12+\tfrac12(n-N)_+$
divided by $n^{3/4}$ decreases before padding and increases
afterwards. At $n=1251$ its numerator is $1231/2<1251/2$;
at $n=1801$ it is $1657/2$, equal to the preceding range's
upper-end numerator. The endpoint comparisons therefore reduce to
\[
 \max\left\{
 \frac{1251}{2\cdot1250^{3/4}},\,
 \frac{1657}{2\cdot1800^{3/4}},\,
 \frac{1999}{2\cdot2400^{3/4}}
 \right\}<3,
\]
as follows by taking fourth powers. The initial matching range
uses the monotonicity of $(n+1)/(2n^{3/4})$ for $n\ge3$;
the orders $1,2$ are immediate.

Every graph constructed has a matching of size $\lfloor n/2\rfloor$,
so Corollary~\ref{prop:matchingmsr} gives
$\msr(H)=\lceil n/2\rceil$. Choose the asymptotic construction whenever its explicit rank
upper bound is at most $3n^{3/4}$, and the finite construction
otherwise. Since $C_*<3$, the first choice applies eventually.
The first construction is Hamiltonian for every sufficiently large
order, so the chosen sequence has this property as well.
This gives one explicit sequence with all the assertions.
\end{proof}

For complete affine lifts with $t\sim aq$, the rank coefficient
$(1+2a)a^{-3/4}$ is minimized at $a=3/2$. The normalized ratio
$\vartheta(G(q,t))/(\rho(G(q,t))|V(G(q,t))|^{1/8})$ has leading
coefficient $a^{7/8}/(2(1+2a))$, whose derivative has the sign of
$7-2a$. Its maximum occurs at $a=7/2$ and is attained asymptotically
by $G(q,\lfloor7q/2\rfloor)$.

The Rosenfeld--Deaett inequality gives
$\rho(G)\le |V(G)|\le2\msr(G)$ on triangle-free graphs;
thus our polynomial semidefinite-rank separation is compatible
with the general linear bounds.

\subsection*{A Rank--Ramsey consequence}

Beniamini et al.\ \cite[Proposition~2.3]{BLS} observed that
replacing graph vertices by cliques preserves complement rank:
$A_H+I=P(A_G+I)P^{\mathsf T}$ for the vertex-repetition matrix $P$.
Following their notation, $R^k(s,t)$ is the least order at which
every graph has an $s$-clique or complement rank at least $t$.
Thus Theorem~\ref{thm:main} gives, for fixed $s\ge3$,
\[
 R^k(s,t)\ge
 \left\lfloor\frac{s-1}{2}\right\rfloor
 \left(\frac{3}{2^{11/3}}+o(1)\right)t^{4/3}
 \qquad(t\to\infty).
\]
Indeed, blow up each vertex of the triangle-free example to a
clique of size $\lfloor(s-1)/2\rfloor$ and use the
asymptotic rank bound. This application is a direct consequence
of the established clique substitution principle, not a new
substitution theorem.

\section*{Declaration of competing interest}

The authors declare that they have no known competing financial interests or
personal relationships that could have appeared to influence the work
reported in this paper.

\section*{Data availability}

No data were used for the research described in this article.

\section*{Declaration on the use of AI}

The authors used ChatGPT  to assist in discussing proof strategies, checking proofs, and improving exposition.

\begingroup\small\raggedright\endgroup

\begin{thebibliography}{99}\setlength{\itemsep}{0pt}\setlength{\parsep}{0pt}
\bibitem{AKS}
M. Ajtai, J. Koml\'os and E. Szemer\'edi,
A note on Ramsey numbers,
\emph{J. Combin. Theory Ser. A} \textbf{29} (1980), 354--360.
\bibitem{BallaLetzterSudakov}
I. Balla, S. Letzter and B. Sudakov,
Orthonormal representations of $H$-free graphs,
\emph{Discrete Comput. Geom.} \textbf{64} (2020), 654--670.
\bibitem{BBIR}
J. Bamberg, A. Bishnoi, F. Ihringer and A. Ravi,
New bounds and constructions for large partial $m$-ovoids and related structures,
\emph{European J. Combin.} \textbf{139} (2027), 104457.
\bibitem{BLS}
G. Beniamini, N. Linial and A. Shraibman,
The Rank--Ramsey problem and the log-rank conjecture,
\emph{Combin. Probab. Comput.} \textbf{35} (2026), 568--597.
\bibitem{Berge}
C. Berge,
Sur le couplage maximum d'un graphe,
\emph{C. R. Acad. Sci. Paris} \textbf{247} (1958), 258--259.
\bibitem{CLL}
S. M. Cioab\u{a}, F. Lazebnik and W. Li,
On the spectrum of Wenger graphs,
\emph{J. Combin. Theory Ser. B} \textbf{107} (2014), 132--139.
\bibitem{CPR}
B. Codenotti, P. Pudl\'ak and G. Resta,
Some structural properties of low-rank matrices related to
computational complexity,
\emph{Theoret. Comput. Sci.} \textbf{235} (2000), 89--107.
\bibitem{Deaett}
L. Deaett,
The minimum semidefinite rank of a triangle-free graph,
\emph{Linear Algebra Appl.} \textbf{434} (2011), 1945--1955.
\bibitem{dKPS}
E. de Klerk, D. V. Pasechnik and A. Schrijver,
Reduction of symmetric semidefinite programs using the regular
$*$-representation,
\emph{Math. Program.} \textbf{109} (2007), 613--624.
\bibitem{Feige}
U. Feige,
Randomized graph products, chromatic numbers, and the Lov\'asz
$\vartheta$-function,
\emph{Combinatorica} \textbf{17} (1997), 79--90.
\bibitem{Haemers}
W. Haemers,
On some problems of Lov\'asz concerning the Shannon capacity of a graph,
\emph{IEEE Trans. Inform. Theory} \textbf{25} (1979), 231--232.
\bibitem{Hall}
P. Hall,
On representatives of subsets,
\emph{J. London Math. Soc.} \textbf{10} (1935), 26--30.
\bibitem{HHR}
F. Harary, S. T. Hedetniemi and R. W. Robinson,
Uniquely colorable graphs,
\emph{J. Combin. Theory} \textbf{6} (1969), 264--270.
\bibitem{Higham}
N. J. Higham,
The power of bidiagonal matrices,
\emph{Electron. J. Linear Algebra} \textbf{40} (2024), 453--474.
\bibitem{KSV}
P. Keevash, B. Sudakov and J. Verstra\"ete,
On a conjecture of Erd\H{o}s and Simonovits: even cycles,
\emph{Combinatorica} \textbf{33} (2013), 699--732.
\bibitem{Kim}
J. H. Kim,
The Ramsey number $R(3,t)$ has order of magnitude $t^2/\log t$,
\emph{Random Structures Algorithms} \textbf{7} (1995), 173--207.
\bibitem{Knuth}
D. E. Knuth,
The sandwich theorem,
\emph{Electron. J. Combin.} \textbf{1} (1994), A1.
\bibitem{Kocbek}
M. Kocbek,
Explicit geometric construction of Ramsey graphs,
\emph{arXiv preprint}
\href{https://arxiv.org/abs/2507.09235v3}{arXiv:2507.09235v3} (2025).
\bibitem{KPR}
A. Kostochka, P. Pudl\'ak and V. R\"odl,
Some constructive bounds on Ramsey numbers,
\emph{J. Combin. Theory Ser. B} \textbf{100} (2010), 439--445.
\bibitem{Lovasz}
L. Lov\'asz,
On the Shannon capacity of a graph,
\emph{IEEE Trans. Inform. Theory} \textbf{25} (1979), 1--7.
\bibitem{LovaszSaks}
L. Lov\'asz and M. Saks,
Lattices, M\"obius functions and communication complexity,
in \emph{29th Annual Symposium on Foundations of Computer Science
(FOCS 1988)}, IEEE Computer Society (1988), 81--90.
\bibitem{Lovett}
S. Lovett,
Communication is bounded by root of rank,
\emph{J. ACM} \textbf{63} (2016), 1:1--1:9.
\bibitem{Nagura}
J. Nagura,
On the interval containing at least one prime number,
\emph{Proc. Japan Acad.} \textbf{28} (1952), 177--181.
\bibitem{Reiman}
I. Reiman,
\"Uber ein Problem von K. Zarankiewicz,
\emph{Acta Math. Acad. Sci. Hungar.} \textbf{9} (1958), 269--273.
\bibitem{Shannon}
C. E. Shannon,
The zero error capacity of a noisy channel,
\emph{IRE Trans. Inform. Theory} \textbf{2} (1956), 8--19.
\bibitem{SudakovTomon}
B. Sudakov and I. Tomon,
Matrix discrepancy and the log-rank conjecture,
\emph{Math. Program.} \textbf{212} (2025), 567--579.
\bibitem{Torres}
L. Torres,
Geometric multiplicity of unitary non-backtracking eigenvalues,
\emph{arXiv preprint}
\href{https://arxiv.org/abs/2205.02004v2}{arXiv:2205.02004v2}
(2026; first posted 2022).
\bibitem{Witte}
D. Witte,
Cayley digraphs of prime-power order are hamiltonian,
\emph{J. Combin. Theory Ser. B} \textbf{40} (1986), 107--112.
\end{thebibliography}
\end{document}